\documentclass[english]{article}
\usepackage{amsmath,amssymb,latexsym,amsthm,mathrsfs,appendix}
\usepackage{color}
\usepackage{scalerel}
\usepackage{graphicx}
\usepackage{caption}
\usepackage{subcaption}
\usepackage{mathtools}
\date{}

\def\theenumi{\arabic{enumi}}

\def\theenumii{\alph{enumii}}
\def\p@enumii{\theenumi.}

\def\theenumiii{\arabic{enumiii}}
\def\p@enumiii{(\theenumi)(\theenumii)}

\def\p@enumiv{\p@enumiii.\theenumiii}
\newtheorem{theorem}{Theorem}[section]

\newtheorem{proposition}[theorem]{Proposition}
\newtheorem{lemma}[theorem]{Lemma}

\theoremstyle{definition}
\newtheorem{example}[theorem]{Example}

\newtheorem{definition}[theorem]{Definition}
\newtheorem{remark}[theorem]{Remark}

 \makeatother

\usepackage{babel}

\begin{document}
\title{Perturbation results concerning Gaussian estimates and hypoellipticity for left-invariant Laplacians on compact groups}
\author{Qi Hou\thanks{%
Partially supported by NSF grant DMS 1707589} \\ {\small Yanqi Lake Beijing Institute of Mathematical Sciences and Applications}\and 
Laurent Saloff-Coste\thanks{ Partially supported by NSF grant DMS 1707589} \\
{\small Department of Mathematics}\\
{\small Cornell University}  }
\maketitle
\begin{abstract}
In this paper we study left-invariant Laplacians on compact connected groups that are form-comparable perturbations of bi-invariant Laplacians. Our results show that Gaussian bounds for derivatives of heat kernels enjoyed by certain bi-invariant Laplacians hold for their form-comparable perturbations. We further show that the parabolic operators associated with such left-invariant Laplacians, in particular, with the bi-invariant Laplacians, are hypoelliptic in various senses.
\end{abstract}
\section{Introduction}
Let $(G,\nu)$\ be a compact connected metrizable group where $\nu$\ is the normalized Haar measure. The main interest of the results developed here is for infinite-dimensional groups $G$, such as the infinite product of compact Lie groups. Symmetric Gaussian convolution semigroups of measures on $G$\ exhibit a rich variety of behaviours, cf. \cite{Elliptic,orange,continuousdensity}. In \cite{Gaussian}, Bendikov and Saloff-Coste showed that (1) when they exist, continuous density functions of symmetric central Gaussian convolution semigroups of measures $(\mu_t)_{t>0}$\ on $G$\ admit spatial derivatives of all orders in certain directions, and (2) the derivatives satisfy certain Gaussian upper bounds if the density function satisfies $\lim\limits_{t\rightarrow 0}t\log{\mu_t(e)}=0$\ (referred to as the (CK$*$) condition). The same authors studied various function spaces on $G$\ with norms involving different combinations of derivatives, and examined hypoellipticity properties related to the generators of such Gaussian semigroups. See \cite{functionspaces,hypoellipticity}. In these works, the semigroup being central, or equivalently, its generator being bi-invariant, plays an essential role. In the present paper we consider symmetric (noncentral) convolution semigroups on $G$\ (refered to as perturbed semigroups) that are comparable to symmetric central Gaussian semigroups in the sense that their associated Dirichlet forms are comparable.

More precisely, any generator $-\Delta$\ of a symmetric central Gaussian semigroup $(\mu_t^\Delta)_{t>0}$\ on $G$\ is a bi-invariant operator of the form $-\Delta=\sum_{i=1}^{\infty}X_i^2$, where $\{X_i\}_i$\ forms a basis of left-invariant vector fields in some proper sense (i.e. a projective basis). In such a projective basis, the generator $-L$\ of any symmetric Gaussian semigroup $(\mu_t^L)_{t>0}$\ is of the form $-L=\sum_{i,j}a_{ij}X_iX_j$, cf. \cite{Bornformula,Bornbasis}. This operator $L$\ is only left-invariant in general. If $L$\ further satisfies that for some $c,C>0$, the following relation between the Dirichlet forms associated with $\Delta,L$\ holds,
    \begin{eqnarray*}
    c\mathcal{E}_\Delta\leq \mathcal{E}_L\leq C\mathcal{E}_\Delta,
    \end{eqnarray*}
we call $L$\ a form-comparable perturbation of $\Delta$. In this paper we only consider those $L$\ that are form-comparable perturbations of bi-invariant Laplacians $\Delta$.

We show that when the semigroup $\mu_t^\Delta$\ admits a continuous density function, denoted again by $\mu_t^\Delta$, so does the semigroup $\mu_t^L$, and the density functions $\mu_t^\Delta$\ and $\mu_t^L$\ both belong to certain smooth function spaces associated with $\Delta$\ and $L$. See Theorem \ref{Tspacethm}. We then show that when $\mu_t^\Delta$\ further satisfies the (CK$*$) condition, (1) both density functions and their derivatives satisfy certain Gaussian estimates, and (2) both $\partial_t+\Delta$\ and $\partial_t+L$\ are hypoelliptic in various senses, see Proposition \ref{prop1} and Theorem \ref{mainthm}. As in \cite{hypoellipticity}, our proof for the hypoellipticity properties follows the general heat kernel/semigroup approach by Kusuoka and Stroock \cite{Kusuoka}, with the help of some additional ideas to overcome the difficulties brought by dropping the bi-invariance assumption. Note that in the bi-invariant operator case, the hypoellipticity of $\partial_t+\Delta$\ is new even though the hypoellipticity of $\Delta$\ has been studied in \cite{hypoellipticity}.

We organize this paper as follows. In Section 2 we briefly introduce the setting and fix notations. In Section 3 we present the main theorem on hypoellipticity, Theorem \ref{mainthm}, and give an example. Section 4 and Section 5 are devoted to the proof of the hypoellipticity theorem, where other theorems regarding heat kernels are presented first and they constitute an important part of the proof. Some details are postponed to the Appendix.

\section{Setting and notation}
This section contains a minimal introduction to the setting of our study. For more details, see \cite{continuousdensity,Gaussiansurvey,brownianmotion} and the references therein.
\subsection{Gaussian semigroups and generators}
Let $G$\ be a compact connected metrizable group with identity element $e$. To define the (projective) Lie algebra of $G$, we view $G$\ as the projective limit of a sequence of Lie groups $G=\varprojlim_{\alpha\in \aleph} G_\alpha$, cf. \cite{groupstructure2,groupstructure}. Here the index set $\aleph$\ is finite or countable; $G_\alpha=G/K_\alpha$\ where $\{K_\alpha\}$\ is a decreasing sequence of compact normal subgroups with $\cap_{\alpha\in \aleph}K_\alpha=\{e\}$; for $\alpha\leq \beta$, the map $\pi_{\alpha,\beta}:G_\alpha\rightarrow G_\beta$\ is the projection map. Denote the projection map $G\rightarrow G_\alpha=G/K_\alpha$\ by $\pi_\alpha$. The Lie algebra of $G$\ (denoted by $\mathfrak{g}$) is defined to be the projective limit of the Lie algebras of $G_\alpha$\ (denoted by $\mathfrak{g}_\alpha$) with projection maps $d\pi_{\alpha,\beta}$. A family $\{X_i\}_{i\in \mathcal{I}}$\ of elements of $\mathfrak{g}$\ is called a \textit{projective basis} of $\mathfrak{g}$, if for each $\alpha\in \aleph$, there is a finite subset $\mathcal{I}_\alpha\subset \mathcal{I}$, such that $d\pi_\alpha(X_i)=0$\ for all $i\notin \mathcal{I}_\alpha$, and $\{d\pi_\alpha(X_i)\}_{i\in \mathcal{I}_\alpha}$\ is a basis of the Lie algebra $\mathfrak{g}_\alpha$. The Lie algebra $\mathfrak{g}$\ admits many projective bases, cf. \cite{Bornbasis,groupstructure}. For future use, let
\begin{eqnarray*}
\mathbb{R}^{(\mathcal{I})}:=\{\xi=(\xi_i)_{i\in \mathcal{I}}:\mbox{all but finitely many entries are zero}\}.
\end{eqnarray*}

For a fixed projective basis $X=\{X_i\}_{i\in \mathcal{I}}$, consider the following second-order left-invariant differential operator given by
\begin{eqnarray}
\label{laplacian}
L=-\sum_{i,j\in \mathcal{I}}a_{ij}X_iX_j,
\end{eqnarray}
where the coefficient matrix $A=(a_{ij})_{\mathcal{I}\times \mathcal{I}}$\ is real symmetric non-negative, i.e., $a_{ij}=a_{ji}$\ are real numbers, and $\sum a_{ij}\xi_i\xi_j\geq 0$\ for all $\xi=(\xi_i)_{i\in \mathcal{I}}\in \mathbb{R}^{(\mathcal{I})}$. We refer to these differential operators as \textit{sub-Laplacians}, note that the matrix $A$\ can be very degenerate. When $A$\ is positive definite, we call $L$\ a Laplacian. These sub-Laplacians $-L$\ are exactly the (infinitesimal) generators of symmetric Gaussian convolution semigroups on $G$\ (denoted by $(\mu_t^L)_{t>0}$), cf. \cite{hypoellipticity,Bornformula}.

More precisely, recall that a family $(\mu_t)_{t>0}$\ of probability measures on $G$\ is called a \textit{symmetric Gaussian (convolution) semigroup}, if it satisfies the following properties
\begin{itemize}
    \item[(i)] (semigroup property) $\mu_t*\mu_s=\mu_{t+s}$, for any $t,s>0$;
    \item[(ii)] (weakly continuous) $\mu_t\rightarrow \delta_e$\ weakly as $t\rightarrow 0$;
    \item[(iii)] (Gaussian) $t^{-1}\mu_t(V^c)\rightarrow 0$\ as $t\rightarrow 0$\ for any neighborhood $V$\ of the identity $e\in G$;
    \item[(iv)] (symmetric) $\check{\mu}_t=\mu_t$\ for any $t>0$, here $\check{\mu}_t$\ is defined by $\check{\mu}_t(V)=\mu_t(V^{-1})$\ for any Borel subset $V\subset G$.
\end{itemize}
We postpone recalling the definitions of convolutions till the end of this subsection.

A sub-Laplacian $-L=\sum a_{ij}X_iX_j$\ (\ref{laplacian}) is the generator of a symmetric Gaussian semigroup $(\mu_t^L)_{t>0}$\ in the sense that for proper smooth functions $f$\ (functions in the Bruhat test function space $\mathcal{B}(G)$, to be defined below),
\begin{eqnarray*}
\sum_{i,j\in \mathcal{I}} a_{ij}X_iX_jf=-Lf=\lim_{t\rightarrow 0}\frac{H^L_tf-f}{t}.
\end{eqnarray*}
Here $(H^L_t)_{t>0}$\ denotes the Markov semigroup associated with $(\mu_t^L)_{t>0}$\ via
\begin{eqnarray*}
H_t^Lf(x)=\int_Gf(xy)\,d\mu_t^L(y),\ \forall f\in C(G),
\end{eqnarray*}
and then extended to $L^2(G,\nu)$. The semigroup $(H_t^L)_{t>0}$\ is self-adjoint and commutes with left translations. See e.g. \cite{hypoellipticity} for more details. Let $(\mathcal{E}_L,\mathcal{F}_L)$\ denote the Dirichlet form associated with $(H_t^L)_{t>0}$\ with domain $\mathcal{F}_L$.

The Bruhat test function space $\mathcal{B}(G)$\ is defined as
\begin{eqnarray*}
\mathcal{B}(G):=\left\{f:G\rightarrow \mathbb{R}:f=\phi\circ\pi_\alpha\ \mbox{for some }\alpha\in \aleph,\ \phi\in C^\infty(G_\alpha)\right\}.
\end{eqnarray*}
Here $C^\infty(G_\alpha)$\ denotes the set of all smooth functions on $G_\alpha$. $\mathcal{B}(G)$\ is independent of the choice of $\{K_\alpha\}_{\alpha\in \aleph}$, cf. \cite{Bruhat,hypoellipticity}. The Bruhat test functions are generalizations of cylindric functions on the infinite torus $\mathbb{T}^\infty$\ (i.e.  smooth functions that depend on only finitely many coordinates). Let $\mathcal{B}'(G)$\ be the dual space of $\mathcal{B}(G)$\ with the strong dual topology, elements of $\mathcal{B}'(G)$\ are called Bruhat distributions. In the next section we introduce more smooth function spaces on $G$\ and their dual spaces as distribution spaces.

Finally, we recall the following definitions of convolutions.
\begin{itemize}
    \item The convolution of any two Borel measures $\mu_1,\mu_2$\ on $G$, $\mu_1*\mu_2$, is a measure defined by
\begin{eqnarray*}
\mu_1*\mu_2(f)=\int_{G\times G}f(xy)\,d\mu_1(x)d\mu_2(y), \ \forall f\in C(G).
\end{eqnarray*}
\item The convolution of any two functions $f,g\in C(G)$\ is the function given by
\begin{eqnarray*}
f*g(x)=\int_Gf(xy^{-1})g(y)\,d\nu(y)=\int_Gf(y)g(y^{-1}x)\,d\nu(y).
\end{eqnarray*}
\item The convolution of a function $f\in C(G)$\ and a Borel measure $\mu$\ is defined as
\begin{eqnarray*}
\mu*f(x)=\int_G f(y^{-1}x)\,d\mu(y),\ f*\mu(x)=\int_G f(xy^{-1})d\mu(y).
\end{eqnarray*}
In this notation, $H_tf(x)=f*\check{\mu}_t(x)$.
\item The convolution of a function $f\in \mathcal{B}(G)$\ and a distribution $U\in \mathcal{B}'(G)$\ is defined by
\begin{eqnarray*}
(f*U)(\phi)=U(\check{f}*\phi),\ (U*f)(\phi)=U(\phi*\check{f}),\ \forall \phi\in \mathcal{B}(G).
\end{eqnarray*}
\end{itemize}

\subsection{Smooth function spaces and distribution spaces}
Let $G$\ be a compact connected metrizable group as above and let $I\subset \mathbb{R}$\ be an open interval. Corresponding to $\mathcal{B}(G)$\ defined in the first subsection, define the Bruhat test function space on $I\times G$, $\mathcal{B}(I\times G)$, as
\begin{eqnarray*}
\mathcal{B}(I\times G):=\left\{f:I\times G\rightarrow \mathbb{R}:f=\phi\circ\widetilde{\pi}_\alpha\ \mbox{for some }\alpha\in \aleph,\ \phi\in C^\infty_c(I\times G_\alpha)\right\}.
\end{eqnarray*}
Here $\widetilde{\pi}_\alpha:I\times G\rightarrow I\times G_\alpha$, $(t,x)\mapsto (t,\pi_\alpha(x))$\ is the projection map. See \cite{Bruhat} for more details.

In the following we review some norms and seminorms and consider the function spaces as completions of $\mathcal{B}(G)$\ with respect to these (semi)norms. See \cite{functionspaces} for detailed discussions of these function spaces. Along the way we define their corresponding spaces on $I\times G$; they serve as generalizations of the classical test function space $\mathcal{D}(\mathbb{R}^n)$\ to spaces of test functions in $I\times G$.
\begin{remark}
We only define the ``smooth'' function spaces in that the completions are w.r.t. the (semi)norms of all orders. The completions w.r.t. (semi)-norms of orders up to some $k\in \mathbb{N}$\ can be correspondingly defined, and denoted accordingly with the ``$\infty$'' superscripts replaced by $k$.
\end{remark}

\paragraph{Spaces $\mathcal{C}^\infty_X(G)$\ and $\mathcal{C}^\infty_{X}(I\times G)$} 
Given any projective basis $\{X_i\}_{i\in \mathcal{I}}$\ and multi-index $l=(l_1,l_2,\cdots,l_k)\in \mathcal{I}^k$\ where $k\in \mathbb{N}$, for any $f\in \mathcal{B}(G)$, denote
\begin{eqnarray*}
X^lf(x)=X_{l_1}X_{l_2}\cdots X_{l_k}f(x)=:D_x^kf(X_{l_1},X_{l_2},\cdots,X_{l_k}).
\end{eqnarray*}
For $k=0$, this is taken as $f$\ itself. Let $\mathcal{C}^\infty_X(G)$\ be the completion of $\mathcal{B}(G)$\ w.r.t. the seminorms
\begin{eqnarray*}
\left\{||X^lf||_{L^\infty(G)}=\sup_{G}|X^lf|:l\in \mathcal{I}^k,\ k\in \mathbb{N}\right\}.
\end{eqnarray*}

Correspondingly, for any open interval $J\Subset I$, let $\mathcal{C}^\infty_{0,X}(J\times G)$\ be the completion of $\mathcal{B}(J\times G)$\ w.r.t. the seminorms
\begin{eqnarray*}
\left\{||\partial_t^mX^lf||_{L^\infty(J\times G)}=\sup_{J\times G}|\partial_t^mX^lf|:l\in \mathcal{I}^k,\ m,k\in \mathbb{N}\right\}.
\end{eqnarray*}
Here the subscript ``$0$'' refers to the vanishing of the function and all its derivatives on the boundary $\partial J$. To clarify notations, if $f=\phi\circ\widetilde{\pi}_\alpha$\ as in the definition of $\mathcal{B}(J\times G)$, then
\begin{eqnarray*}
\partial_t^mX_{l_1}\cdots X_{l_k}f(t,x)=\left(\partial_t^m d\pi_\alpha(X_{l_1})\cdots d\pi_\alpha(X_{l_k})\phi\right)(t,\pi_\alpha(x)).
\end{eqnarray*}

The spaces are equipped with the topology defined by the seminorms above. Note that for any two open intervals $I_1\Subset I_2$, functions in $\mathcal{C}^\infty_{0,X}(I_1\times G)$\ can be naturally extended by $0$\ to $I_2\times G$, thus $\mathcal{C}^\infty_{0,X}(I_1\times G)\subset \mathcal{C}^\infty_{0,X}(I_2\times G)$.

Finally, define $\mathcal{C}^\infty_{X}(I\times G)$\ as the inductive limit of the spaces $\{\mathcal{C}^\infty_{0,X}(J\times G)\}_{J\Subset I}$.

\paragraph{Spaces $\mathcal{S}^\infty_X(G)$\ and $\mathcal{S}^\infty_{X}(I\times G)$} 
For any $k\in \mathbb{N}$, $f\in \mathcal{B}(G)$, define a function $|D^kf|_X:G\rightarrow [0,+\infty]$\ given by
\begin{eqnarray*}
\lefteqn{\left|D^kf\right|_X(x)=\left|D_x^kf\right|_X}\\
&&:=\left(\sum_{l\in \mathcal{I}^k}|D^k_xf(X_{l_1},X_{l_2},\cdots,X_{l_k})|^2\right)^{1/2}=\left(\sum_{l\in \mathcal{I}^k}X^lf(x)\right)^{1/2}.
\end{eqnarray*}
When $k=1$, $|D^1_xf|^2_X=\sum_{i\in \mathcal{I}} |X_if(x)|^2$\ is the square of the gradient of $f$, often denoted by $\Gamma(f,f)$\ in the literature, see e.g. \cite{functionspaces}. Let $\mathcal{S}^\infty_X(G)$\ be the completion of $\mathcal{B}(G)$\ w.r.t. the norms $\{S^k_X\}_{k\in \mathbb{N}}$\ defined as 
\begin{eqnarray*}
S^k_X(f):=\sup_{m\leq k}\left|\left|\left|D^mf\right|_X\right|\right|_{L^\infty(G)}=\sup_{0\leq m\leq k}\sup_{x\in G}\left|D^m_xf\right|_X.
\end{eqnarray*}
\begin{remark}
\label{LXsubscripts}
It is shown in \cite{functionspaces} that for any two projective bases $\{X_i\}$\ and $\{Z_i\}$\ such that $\sum X_i^2=\sum Z_i^2$, for all $x\in G$\ and $k\in \mathbb{N}$, $|D^k_xf|_X=|D^k_xf|_Z$\ and $S^k_X(f)=S^k_Z(f)$. In other words, the space $\mathcal{S}^\infty_X(G)$\ depends on the projective basis $\{X_i\}_{i\in\mathcal{I}}$\ through the sum of squares operator given by the basis, i.e., through
\begin{eqnarray*}
L=-\sum_{i\in \mathcal{I}} X_i^2.
\end{eqnarray*}
Since $\{X_i\}$\ is taken as a projective basis, $L$\ is what we called a Laplacian earlier. Below, we write $|D^k_xf|_L$\ and $\mathcal{S}^\infty_L(G)$\ instead of $|D^k_xf|_X$\ and $\mathcal{S}^\infty_X(G)$.
\end{remark}
For any projective basis $X=\{X_i\}_{i\in \mathcal{I}}$\ and $L=-\sum X_i^2$, for any open interval $J\Subset I$, the space $\mathcal{S}^\infty_{0,X}(J\times G)=\mathcal{S}^\infty_{0,L}(J\times G)$\ is defined as the completion of $\mathcal{B}(J\times G)$\ w.r.t. the (semi)norms
\begin{eqnarray*}
\left\{S^{k,p}_L(J\times G,\,f):=\sup_{0\leq a\leq p}\sup_{t\in J}S^k_L(\partial_t^af)=\sup_{\substack{0\leq a\leq p\\0\leq m\leq k}}\sup_{(t,x)\in J\times G}\left|D^m_x(\partial_t^af)\right|_L\right\}_{k,p\in \mathbb{N}}.
\end{eqnarray*}
Here writing $J\times G$\ in $S^{k,p}_L(J\times G,\,f)$\ is to clarify the set on which the supremum is taken.

Finally, the space $\mathcal{S}^\infty_{L}(I\times G)$\ is defined as the inductive limit of the spaces $\{\mathcal{S}^\infty_{0,L}(J\times G)\}_{J\Subset I}$.

\paragraph{Spaces $\mathcal{T}^\infty_L(G)$\ and $\mathcal{T}_L^\infty(I\times G)$}
It is not known if $\mathcal{S}^\infty_L(G)$\ is contained in the $C(G)$-domain of $L$, see the comment in \cite{functionspaces}. There, to address this problem, the authors introduced the space $\mathcal{T}_L^\infty(G)$\ which we now recall the definition of.

Fix any projective basis $X=\{X_i\}_{i\in \mathcal{I}}$\ and let $L=-\sum_{i\in \mathcal{I}}X_i^2$. For any $k\in \mathbb{N}$\ and $\lambda=(\lambda_0,\lambda_1,\cdots,\lambda_k)\in \mathbb{N}^{k+1}$, for any $f\in \mathcal{B}(G)$, let
\begin{eqnarray*}
|D^{k,\lambda}_xf|_X=|D^{k,\lambda}_xf|_L:=\left(\sum_{l\in \mathcal{I}^k}|L^{\lambda_0}X_{l_1}L^{\lambda_1}X_{l_2}L^{\lambda_2}\cdots X_{l_k}L^{\lambda_k}f(x)|^2\right)^{1/2}.
\end{eqnarray*}
The subscripts can be written as $X$\ or $L$\ because the quantity depends only on the sum of squares associated with the projective basis $X$, see also Remark \ref{LXsubscripts}. Now we take supremum over $k$\ and $\lambda$\ to build a norm. We first introduce the following two abbreviated notations.
\begin{itemize}
    \item For any $l\in \mathcal{I}^k$\ and $\lambda\in \mathbb{N}^{k+1}$, set
\begin{eqnarray*}
P_L^{l,\lambda}f:=L^{\lambda_0}X_{l_1}L^{\lambda_1}X_{l_2}L^{\lambda_2}\cdots X_{l_k}L^{\lambda_k}f.
\end{eqnarray*}
For any bi-invariant Laplacian $\Delta$, $P_\Delta^{l,\lambda}f=X_{l_1}X_{l_2}\cdots X_{l_k}\Delta^{\lambda_0+\cdots+\lambda_k}f$.
\item For any $k,m\in \mathbb{N}$, let $\Lambda(k,m)$\ denote the set of all possible $(k+1)-$tuples of integers that sum up to $m$, i.e.,
\begin{eqnarray*}
\Lambda(k,m):=\{\lambda=(\lambda_0,\lambda_1,\cdots,\lambda_k):\lambda_i\in \mathbb{N},\ \sum_{i=0}^{k}\lambda_i=m\}.
\end{eqnarray*}
\end{itemize}
Using these notations, for any $N\in \mathbb{N}$, define
\begin{eqnarray*}
\lefteqn{M^N_X(f)=M^N_L(f):=}\\
&&\hspace{-.2in}\sup_{x\in G}\sup_{\substack{k,m\in \mathbb{N}\\k+2m\leq N}}\sup_{\lambda\in \Lambda(k,m)}\left|D_x^{k,\lambda}f\right|_L=\sup_{x\in G}\sup_{\substack{k,m\in \mathbb{N}\\k+2m\leq N}}\sup_{\lambda\in \Lambda(k,m)}\left(\sum_{l\in \mathcal{I}^k}|P_L^{l,\lambda}f(x)|^2\right)^{1/2}.
\end{eqnarray*}

Define $\mathcal{T}^\infty_L(G)$\ as the completion of $\mathcal{B}(G)$\ w.r.t. the norms $\{M^N_L\}_{N\in \mathbb{N}}$. For any open interval $J\Subset I$, define $\mathcal{T}_{0,L}^\infty(J\times G)$\ as the completion of $\mathcal{B}(J\times G)$\ w.r.t. the set of (semi)norms
\begin{eqnarray*}
\lefteqn{\left\{M^{N,p}_L(J\times G,\,f):=\sup_{0\leq a\leq p}\sup_{t\in J}M^N_L(\partial_t^af)\right.}\\
&&\left.=\sup_{\substack{0\leq a\leq p\\k,m\in \mathbb{N}\\k+2m\leq N}}\sup_{(t,x)\in J\times G}\sup_{\lambda\in \Lambda(k,m)}\left(\sum_{l\in \mathcal{I}^k}|P_L^{l,\lambda}\partial_t^af(t,x)|^2\right)^{1/2}\right\},
\end{eqnarray*}
where $N,p\in \mathbb{N}$. As before, the space $\mathcal{T}_{L}^\infty(I\times G)$\ is defined as the inductive limit of $\{\mathcal{T}_{0,L}^\infty(J\times G)\}_{J\Subset I}$.
\paragraph{Distribution spaces $\mathcal{T}'_L(G)$\ and $\mathcal{T}'_L(I\times G)$} Let $\mathcal{T}'_L(G)$\ and $\mathcal{T}'_L(I\times G)$\ be the dual spaces of $\mathcal{T}^\infty_L(G)$\ and $\mathcal{T}^\infty_{L}(I\times G)$, respectively, equipped with the strong dual topology. Recall that, it means that for any $U\in \mathcal{T}'_L(I\times G)$, for any precompact open subset $J\Subset I$, there exist some $C(J)>0$\ and $N(J),p(J)\in \mathbb{N}_+$\ (written as $N,p$\ below) such that for any $f\in \mathcal{T}^\infty_{0,L}(J\times G)$,
\begin{eqnarray*}
|U(f)|\leq C(J)M_L^{N,p}(f).
\end{eqnarray*}

\begin{remark}
In this paper, we deal with distributional solutions $U$\ of the heat equation $(\partial_t+L)U=F$. Because $\mathcal{C}^\infty_X$\ and $\mathcal{S}^\infty_L$\ are not (may not be) contained in the domain of $L$, to make sense of $LU$\ we do not consider the dual spaces of $\mathcal{C}^\infty_X$\ or $\mathcal{S}^\infty_L$.
\end{remark}

To summarize, the spaces above satisfy the relation
\begin{eqnarray*}
\mathcal{B}(I\times G)\subset \mathcal{T}^\infty_{L}(I\times G)\subset \mathcal{S}^\infty_{L}(I\times G)\subset \mathcal{C}^\infty_{X}(I\times G)\subset \mathcal{T}'_L(I\times G)\subset \mathcal{B}'_L(I\times G);
\end{eqnarray*}
the same inclusion relation holds for the corresponding spaces on $G$.

%Now we recall and fix notations for convolutions involving both time and space.

By the previously reviewed definitions of convolutions, the convolution of any two functions $f,g\in C(\mathbb{R}\times G)$\ (at least one function with compact support) is the function given by
\begin{eqnarray}
\label{convdef}
f* g(s,x)=\int_\mathbb{R}\int_Gf(t,z)g(s-t,z^{-1}x)\,dtd\nu(z).
\end{eqnarray}
{\bf Notation}. To emphasize the difference between convolution on $G$\ and convolution on $I\times G$, in the rest of the paper we write $\star$\ instead of $*$\ for convolutions involving both time and space. For example, we write $f\star g$\ for (\ref{convdef}).

In this paper we also use the following convolution. Let $f$\ be any function satisfying that $f,\check{f}\in \mathcal{T}_L^\infty(\mathbb{R}\times G)$, here $\check{f}(s,x)=f(-s,x^{-1})$. Let $U$\ be any distribution in $\mathcal{T}_L'(I\times G)$\ with compact support in $I\times G$. We may define their convolutions as
\begin{eqnarray*}
(f\star U)(\phi)=U(\check{f}\star\phi),\ (U\star f)(\phi)=U(\phi\star \check{f}),\  \forall \phi\in \mathcal{T}^\infty_L(I\times G),
\end{eqnarray*}
These convolutions are well-defined, since (1) both $\check{f}\star\phi$\ and $\phi\star\check{f}$\ belong to $\mathcal{T}_L(\mathbb{R}\times G)$; (2) as $U$\ has compact support in $I\times G$, there is some function $\eta\in \mathcal{T}_L(I\times G)$\ with compact support such that $U=\eta U$.

\subsection{Bi-invariant Laplacians and their perturbations}
By definition, a symmetric Gaussian semigroup $(\mu_t)_{t>0}$\ is called \textit{central}, if $\check{\mu}_t=\mu_t$\ for all $t>0$. Central semigroups commute with any other factor in convolutions. Their generators are bi-invariant Laplacians. In this paper we use the symbol $-\Delta$\ to denote the generator of a central symmetric Gaussian semigroup; the symbol $-L$\ stands for general left-invariant sub-Laplacians that may or may not be bi-invariant. As a bi-invariant operator, the operator $\Delta$\ satisfies
\begin{eqnarray*}
\Delta Zf=Z\Delta f
\end{eqnarray*}
for any left-invariant vector field $Z\in \mathfrak{g}$\ and smooth funciton $f\in \mathcal{B}(G)$.

Let $L,P$\ be two sub-Laplacians (i.e. generators of some symmetric Gaussian semigroups). We say in this paper that $L,P$\ are \textit{form-comparable perturbations} of each other, if their corresponding Dirichlet forms $(\mathcal{E}_P,\mathcal{F}_P)$\ and $(\mathcal{E}_L,\mathcal{F}_L)$\ are comparable. That is, there exist some $c,C>0$, s.t. for any $f\in \mathcal{B}(G)$,
\begin{eqnarray}
\label{comparableforms}
c\mathcal{E}_P(f,f)\leq \mathcal{E}_L(f,f)\leq C\mathcal{E}_P(f,f).
\end{eqnarray}
Note that (\ref{comparableforms}) implies that $\mathcal{F}_L=\mathcal{F}_P$. In the present paper we study perturbations of bi-invariant Laplacians, these perturbations have to be Laplacians themselves.

Given a bi-invariant Laplacian $\Delta$\ and a left-invariant Laplacian $L$, we can mix them up to define the function/distribution spaces $\mathcal{T}^\infty_{\Delta,L}(I\times G)$, $\mathcal{T}'_{\Delta,L}(I\times G)$\ as follows. Suppose $\Delta=-\sum_{i\in \mathcal{I}} X_i^2$. The space $\mathcal{T}^\infty_{\Delta,L}(I\times G)$\ is the inductive limit of $\{\mathcal{T}^\infty_{0,\Delta,L}(J\times G)\}_{J\Subset I}$, where $\mathcal{T}^\infty_{0,\Delta,L}(J\times G)$\ for each open interval $J\Subset I$\ is defined as the completion of $\mathcal{B}(J\times G)$\ with respect to the (semi)norms
\begin{eqnarray}
\label{mixednorm}
&&\hspace{-.5in}M^{N,p}_{\Delta,L}(J\times G,\,f):=\notag\\
&&\hspace{-.5in}\sup_{\substack{0\leq a\leq p\\2b+k+2m\leq N\\(t,x)\in J\times G}}\sup_{\substack{\lambda\in \Lambda(k,m)\\\lambda=(\lambda_0,\cdots,\lambda_k)}}\left(\sum_{l\in \mathcal{I}^k}|L^{\lambda_0}X_{l_1}L^{\lambda_1}\cdots X_{l_k}L^{\lambda_k}\Delta^{b}\partial_t^af(t,x)|^2\right)^{1/2}.
\end{eqnarray}
$\mathcal{T}'_{\Delta,L}(I\times G)$\ is the strong topological dual of $\mathcal{T}^\infty_{\Delta,L}(I\times G)$.
\begin{remark}
\label{normequiv}
Suppose $L=-\sum_{i\in \mathcal{I}}Y_i^2$\ for another projective basis $\{Y_i\}$. In the Appendix, we show that the norms defined using $\{Y_i\}$\ instead of $\{X_i\}$\ or using both $\{X_i\}$\ and $\{Y_i\}$\ is equivalent to the norms (\ref{mixednorm}) above.
\end{remark}

\subsection{Properties of the spaces}
\label{convorderremark}
We record here some properties of the above function spaces that are particularly useful in the present paper. See \cite{functionspaces} for more details. The first three items hold for function spaces associated with general left-invariant sub-Laplacians $L=-\sum Y_i^2$.
\begin{itemize}
    \item[(i)] All the function spaces on $G$\ introduced above ($\mathcal{B}(G)$, $\mathcal{C}^k_Y(G)$, $\mathcal{S}^k_L(G)$, $\mathcal{T}_L^k(G)$, $k\in \mathbb{N}\cup\{\infty\}$) are algebras for pointwise multiplication. Cf. \cite[Sections 2,\,3]{functionspaces}. Based on this fact it is straightforward to check that the corresponding function spaces on $I\times G$\ are algebras for pointwise multiplication as well.
    \item[(ii)] Let $\mathfrak{S}(G)$\ be any of the function spaces $\mathcal{B}(G)$, $\mathcal{C}^k_Y(G)$, $\mathcal{S}^k_L(G)$, and $\mathcal{T}_L^k(G)$. For any Borel measure $\mu$, for any function $f\in \mathfrak{S}(G)$, $\mu*f\in \mathfrak{S}(G)$, and the map (convolution with $\mu$\ on the left)
    \begin{eqnarray*}
    \mu*: \mathfrak{S}(G)\rightarrow \mathfrak{S}(G),\ f\mapsto \mu*f
    \end{eqnarray*}
    is continuous. Let $||\cdot||$\ denote any norm involved in the definitions of these spaces, then $||\mu*f||\leq |\mu|||f||$\ where $|\mu|$\ is the total mass of $\mu$. The same statement holds for $\mu\star:\mathfrak{S}(I\times G)\rightarrow \mathfrak{S}(I\times G)$, if the measure $\mu$\ is on $I\times G$.
    \item[(iii)] For any sequence $\phi_n\in L^1(G)$\ with $\phi_n\rightarrow \delta_e$\ weakly as $n\rightarrow \infty$, for any $f\in \mathfrak{S}(G)$\ as in the previous item, the sequence $f_n:=\phi_n*f$\ converges to $f$\ in $\mathfrak{S}(G)$. The same statement holds for any sequence $\phi_n\in L^1((0,1)\times G)$\ with $\phi_n\rightarrow \delta_{(0,e)}$\ weakly and any $f\in \mathfrak{S}(I\times G)$.
\end{itemize}

When $\mathfrak{S}=\mathcal{B}(G)$\ or $\mathcal{B}(I\times G)$, conclusions of (ii) and (iii) hold too for convolutions on the right. When the space $\mathfrak{S}$\ is any of the other three types, however, due to the fact that the differential operators are left-invariant (they go to the rightmost function in a convolution), the order of the convolutions in (ii) and (iii) is crucial. For instance, even if the approximation of identity $\{\phi_n\}$\ in (iii) consists of very nice functions that are in $\mathcal{B}(G)$\ or $\mathcal{B}(I\times G)$, for an arbitrary function $f$\ in $\mathcal{T}^k_L(G)$\ or $\mathcal{T}^{k,p}_L(I\times G)$, it is not clear in general if $f*\phi_n$\ converges to $f$\ in the $\mathcal{T}$-space. On the other hand, using a trick mentioned in \cite{functionspaces,hypoellipticity} regarding making use of certain right-invariant vector fields, one can make some affirmative statements for convolutions on the right. We list a few relevant results here. As above, the results are true for spaces of each finite order, and the underlying space can be $G$\ or $I\times G$. For simplicity we only write the superscript $\infty$\ and do no specify the underlying space.
\begin{itemize}
    \item[(iv)] For a bi-invariant Laplacian (denoted by $\Delta$), conclusions of (ii) and (iii) hold for convolutions on the right when $\mathfrak{S}=\mathcal{S}^\infty_\Delta$\ or $\mathcal{T}_\Delta^\infty$.
    \item[(v)] For a left-invariant Laplacian $L$\ that is form comparable to a bi-invariant Laplacian $\Delta$, by Lemma \ref{comparablelemma} in the Appendix, $\mathcal{S}^\infty_L=\mathcal{S}^\infty_\Delta$. Hence (ii) and (iii) hold for convolutions on the right when $\mathfrak{S}=\mathcal{S}_L^\infty$.
\end{itemize}

Finally we record a result for $\mathcal{C}$-type spaces. For any left-invariant vector field $Z$, let $\breve{Z}$\ be the right-invariant vector field on $G$\ satisfying that $\breve{Z}(e)=Z(e)$. For any projective basis $Y=\{Y_i\}$\ of left-invariant vector fields, in \cite{hypoellipticity} the authors considered a right-invariant version of the $\mathcal{C}^k_Y$\ space ($k\in \mathbb{N}$), denoted by $\mathcal{RC}^k_Y$, which is analogously defined as the completion of the Bruhat space w.r.t. the seminorms $||\breve{Y}^lf||_\infty$. A special case is when $Y$\ is a so-called special projective basis, which roughly speaking requires that each $Y_i$\ is a finite linear combination of $\{\breve{Y_j}\}$: $Y_i=\sum_{j\in J(i)} a_{ij}\breve{Y}_j$, where each $a_{ij}\in \mathcal{B}(G)$\ and $J(i)$\ is a finite index set, and vice versa for each $\breve{Y_i}$. See \cite{hypoellipticity} for more details. When $Y$\ is a special projective basis,  $\mathcal{C}^k_Y=\mathcal{RC}^k_Y$. For a bi-invariant Laplacian, $\Delta$, there exists a special projective basis $X=\{X_i\}$\ such that $\Delta=-\sum X_i^2$. Hence the next item follows.
\begin{itemize}
    \item[(vi)] For a bi-invariant Laplacian, $\Delta$, and any special projective basis, $X=\{X_i\}$, such that $\Delta=-\sum X_i^2$, (ii) and (iii) hold when $\mathfrak{S}=\mathcal{C}^\infty_X=\mathcal{RC}^\infty_X$.
\end{itemize}

\section{Statement of the main results}
As preparations for stating the main results, we first review some definitions mentioned in the introduction.
\begin{definition}[Property (CK$*$)]
Let $(\mu_t)_{t>0}$\ be a symmetric Gaussian convolution semigorup. We say that $(\mu_t)_{t>0}$\ satisfies Property (CK$*$), if for any $t>0$, $\mu_t$\ admits a continuous density $\mu_t(\cdot)$\ w.r.t. the Haar measure $\nu$, and the density function satisfies that
    \begin{eqnarray}
    \label{CK*}
    \lim_{t\rightarrow 0^+}t\log{\mu_t(e)}=0.
    \end{eqnarray}
\end{definition}
\begin{remark}
\label{CKremark}
Let $\Delta$\ be a bi-invariant Laplacian satisfying Property (CK$*$). Let $L$\ be a left-invariant Laplacian that is form-comparable to $\Delta$. Then there are some $\beta,\gamma>0$\ such that $\mathcal{E}_{\beta \Delta}\leq \mathcal{E}_L \leq \mathcal{E}_{\gamma \Delta}$. Denote the two operators' associated (symmetric) semigroups as $\mu_t^\Delta$, $\mu_t^L$, respectively. Then $\mu_t^{L}=\mu_t^{L-\beta \Delta}*\mu_t^{\beta \Delta}$, indicating that $\mu_t^L$\ admits a continuous density as $\mu_t^{\beta \Delta}$\ does. Moreover, because $\mu_t^L$\ is symmetric and $||\mu_t^L||_{L^\infty(G)}=\mu_t^L(e)$,
\begin{eqnarray*}
\mu_t^{L}(e)=\mu_t^{L-\beta \Delta}*\mu_t^{\beta \Delta}(e)\leq \mu_t^{\beta \Delta}(e).
\end{eqnarray*}
Hence $\mu_t^{L}$\ also satisfies (CK$*$).
\end{remark}
\begin{definition}[Hypoellipticity]
Let $L$\ be a left-invariant sub-Laplacian on $G$\ and $I\subset \mathbb{R}$\ be an open interval. Let $\mathfrak{A}$\ be a space of distributions in time and space. Let $\mathfrak{S}$\ be a space of continuous functions on $I\times G$. The associated parabolic operator $\partial_t+L$\ is said to be $\mathfrak{A}$-$\mathfrak{S}$-hypoelliptic, if for any $U\in \mathfrak{A}$\ and $F\in \mathcal{B}'(I\times G)$\ such that
\begin{eqnarray*}
(\partial_t+L)U=F\ \mbox{in }\mathcal{B}'(I\times G),
\end{eqnarray*}
and for any open subset $J\times\Omega\subset I\times G$\ such that
\begin{eqnarray*}
\forall \varphi\in \mathcal{B}_c(J\times \Omega),\ \varphi F\in \mathfrak{S},
\end{eqnarray*}
$U$\ satisfies that
\begin{eqnarray*}
\forall \varphi\in \mathcal{B}_c(J\times \Omega),\ \varphi U\in \mathfrak{S}.
\end{eqnarray*}
\end{definition}
Here $\mathcal{B}_c(J\times \Omega)$\ consists of functions in $\mathcal{B}(I\times G)$\ with compact support in $J\times \Omega$. For example, $\mathfrak{A}$\ can be $\mathcal{T}_{\Delta,L}'(I\times G)$, and $\mathfrak{S}$\ can be $C(I\times G)$. 
\begin{remark}
\label{hypoequiv}
Observe that for any $N,p\in \mathbb{N}\cup\{\infty\}$, functions in $\mathcal{S}^{N,p}_L(I\times G)$\ have compact supports in $I\times G$, whereas functions in $C^p(I\rightarrow \mathcal{S}^N_L(G))$\ need not. Nevertheless, by definition, being $\mathcal{T}'_{L}(I\times G)-\mathcal{S}_L^{N,p}(I\times G)-$hypoelliptic is equivalent to being $\mathcal{T}'_{L}(I\times G)-C^p\left(I\rightarrow \mathcal{S}_L^{N}(G)\right)-$hypoelliptic. The same is true for $\mathcal{C}$- and $\mathcal{T}$-type spaces.
\end{remark}
The main results of this paper regarding hypoellipticity are summarized in the following theorem.
\begin{theorem}
\label{mainthm}
Let $(G,\nu)$\ be a compact connected metrizable group with normalized Haar measure $\nu$, let $I\subset \mathbb{R}$\ be an open interval. Let $(\mu_t^\Delta)_{t>0}$\ be a symmetric central Gaussian convolution semigroup on $G$\ with generator $-\Delta$\ and write $\Delta=-\sum X_i^2$\ for some special projective basis $\{X_i\}_{i\in \mathcal{I}}$. Suppose $(\mu_t^\Delta)_{t>0}$\ satisfies Property (CK$*$). Let $L$\ be any form-comparable perturbation of $\Delta$\ in the sense of (\ref{comparableforms}). Then the parabolic operator $\partial_t+L$\ is $\mathcal{T}'_{\Delta,L}-\mathfrak{S}-$hypoelliptic. Here $\mathfrak{S}$\ can be $\mathcal{T}_\Delta^{N,p}$, $\mathcal{S}^{N,p}_L=\mathcal{S}^{N,p}_\Delta$, $\mathcal{C}_X^{N,p}$, where $N,p\in \mathbb{N}\cup\{\infty\}$; all spaces are on $I\times G$. 
\end{theorem}

\begin{remark}
Together with Remark \ref{hypoequiv}, this gives the fact that $L$\ is $\mathcal{T}'_{\Delta,L}(I\times G)-C(I\times G)-$hypoelliptic, and more generally, $L$\ is $\mathcal{T}'_{\Delta,L}(I\times G)-C^p(I\rightarrow \mathcal{C}^N_X(G))-$hypoelliptic for $N,p\in \mathbb{N}\cup \{\infty\}$.
\end{remark}
\begin{remark}
Note that in Theorem \ref{mainthm}, the only $\mathcal{T}$-type space we consider in hypoellipticity is the $\mathcal{T}_\Delta$\ space, our proof does not treat the $\mathcal{T}_L$\ and $\mathcal{T}_{\Delta,L}$\ cases.
\end{remark}

Recall that given $\Delta=-\sum X_i^2$, all other sub-Laplacians on $G$\ are of the form $L_A=-\sum a_{ij}X_iX_j$\ for some matrix $A=(a_{ij})$. The next example provides a simple well-known condition on the matrix $A$\ such that $L_A$\ is a form-comparable perturbation of $\Delta$\ in the sense of (\ref{comparableforms}).
\begin{example} For any number $0<\epsilon<1$, we say that a matrix $A=(a_{ij})_{i,j\in \mathcal{I}}$\ is $\epsilon-$diagonally dominant, if for any $i\in \mathcal{I}$,
\begin{eqnarray*}
\epsilon|a_{ii}|>\sum_{j\neq i}|a_{ij}|.
\end{eqnarray*}
Let $(G,\nu)$\ be a compact connected group and $\Delta=-\sum_{i\in \mathcal{I}}X_i^2$\ be a bi-invariant Laplacian on $G$\ as before. Let $L_A:=-\sum_{i,j\in \mathcal{I}}a_{ij}X_iX_j$\ be a Laplacian on $G$\ where $A=(a_{ij})$\ is an $\epsilon-$diagonally dominant matrix for some $0<\epsilon<1$. Then $L_A$\ is a form-comparable perturbation of $\Delta$. Theorem \ref{mainthm} applies if $(\mu_t^\Delta)_{t>0}$\ satisfies Property (CK$*$).
\end{example}

\section{Proof of Theorem \ref{mainthm} - properties of the perturbed heat kernel}
In this section we prove that certain regularity properties of the heat kernel are preserved by form-comparable perturbations of bi-invariant Laplacians. These results are useful in the proof of Theorem \ref{mainthm} and are interesting by themselves. For simplicity we write $\mathcal{T}_{\Delta,L}$\ for $\mathcal{T}^\infty_{\Delta,L}$.
\begin{theorem}
\label{Tspacethm}
Let $G$\ be a compact connected metrizable group. Let $(\mu_t^\Delta)_{t>0}$\ be a symmetric central Gaussian semigroup on $G$\ with generator $-\Delta$. Assume that $(\mu_t^\Delta)_{t>0}$\ admits a continuous density function, denoted again by $\mu_t^\Delta$. Then for any left-invariant Laplacian $L$\ that is form-comparable to $\Delta$, its corresponding semigroup $(\mu_t^L)_{t>0}$\ admits a continuous density function (denoted again by $\mu_t^L$), and the density functions $\mu_t^\Delta$\ and $\mu_t^L$\ both belong to $C^\infty((0,\infty)\rightarrow \mathcal{T}_{\Delta,L}(G))$.
\end{theorem}
Observe that this theorem implies that for any bump function $b(t)$, say $b(t)\in C^\infty_c((0,1))$, the product functions $b(t)\mu_t^\Delta$\ and $b(t)\mu_t^L$\ belong to $\mathcal{T}_{\Delta,L}((0,1)\times G)$.
\begin{proof} We first show that $\mu_t^\Delta\in C^\infty((0,\infty)\rightarrow \mathcal{T}_{\Delta,L}(G))$. For convenience, we sometimes write the density function $\mu_t^\Delta(x)$\ as $\mu^\Delta(t,x)$. Let $\{X_i\}_{i\in \mathcal{I}}$, $\{Y_i\}_{i\in \mathcal{I}}$\ be two projective bases such that $\Delta=-\sum X_i^2$, $L=-\sum Y_i^2$. By \cite{Gaussian}, $\mu_t^\Delta\in C^\infty((0,\infty)\rightarrow\mathcal{T}_\Delta(G))$. Because the seminorms of $\mathcal{T}_{\Delta,L}$\ can be given using either $\{X_i\}$\ or $\{Y_i\}$\ (see Remark \ref{normequiv}), for simplicity of notations we use $\{Y_i\}$\ here, and because $\partial_t\mu_t^\Delta=-\Delta\mu_t^\Delta$\ in $C(G)$, to show that $\mu_t^\Delta\in C^\infty((0,\infty)\rightarrow \mathcal{T}_{\Delta,L}(G))$, it suffices to show that for any $N,p\in \mathbb{N}$, for any $t>0$,
\begin{eqnarray*}
M^{N}_L(\partial_t^p\mu^\Delta_t)=\sup_{\substack{k,m\in \mathbb{N}\\k+2m\leq N}}\sup_{x\in G}\sup_{\lambda\in \Lambda(k,m)}\left(\sum_{l\in \mathcal{I}^k}|P_L^{l,\lambda}\partial_t^p\mu_t^\Delta(x)|^2\right)^{1/2}<\infty.
\end{eqnarray*}
%By Lemma \ref{comparablelemma}, $S^k_{\Delta}$\ and $S^k_{L}$\ are equivalent for any $k\in \mathbb{N}$, so
%\begin{eqnarray*}
%S^k_Y(\partial_t^p\mu_t^\Delta)=S^k_L(\partial_t^p\mu_t^\Delta)<\infty
%\end{eqnarray*}
%for any $p\in \mathbb{N}$. So it remains to estimate the norm involving the $L$\ derivatives.

We now compute $P_L^{l,\lambda}\mu_t^\Delta(x)=L^{\lambda_0}Y_{l_1}L^{\lambda_1}\cdots Y_{l_k}L^{\lambda_k}\mu_t^\Delta(x)$. First, let $\epsilon:=\frac{1}{2k+1}$. For any fixed $\alpha$\ with $0<\alpha<\epsilon C^{-1}=\frac{1}{(2k+1)C}$\ ($C$\ as in $\mathcal{E}_L\leq C\mathcal{E}_{\Delta}$), we have $\mathcal{E}_{\epsilon\Delta-\alpha L}\geq \delta \mathcal{E}_\Delta$\ for some small $\delta>0$. Thus as $\Delta$\ commutes with $L$, $\mu_t^\Delta$\ can be decomposed as
\begin{eqnarray*}
\mu_t^\Delta=\mu_t^{(1-\epsilon)\Delta}*\mu_t^{\epsilon\Delta-\alpha L}*\mu_t^{\alpha L}.
\end{eqnarray*}
Because $\mu_t^\Delta$\ commutes with any function in convolution,
\begin{eqnarray*}
\lefteqn{Y_{l_k}L^{\lambda_k}\mu_t^\Delta=Y_{l_k}\left(\mu_t^{(1-\epsilon)\Delta}*\mu_t^{\epsilon\Delta-\alpha L}*L^{\lambda_k}\mu_t^{\alpha L}\right)}\\
&=&Y_{l_k}\left(\mu_t^{(1-2\epsilon)\Delta}*\mu_t^{\epsilon\Delta-\alpha L}*L^{\lambda_k}\mu_t^{\alpha L}*\mu_t^{\epsilon\Delta}\right)\\
&=&\mu_t^{(1-2\epsilon)\Delta}*\mu_t^{\epsilon\Delta-\alpha L}*L^{\lambda_k}\mu_t^{\alpha L}*Y_{l_k}\mu_t^{\epsilon\Delta}.
\end{eqnarray*}
So
\begin{eqnarray*}
\lefteqn{P_L^{l,\lambda}\mu_t^\Delta(x)=L^{\lambda_0}Y_{l_1}L^{\lambda_1}\cdots Y_{l_k}L^{\lambda_k}\mu_t^\Delta(x)}\\
&=&L^{\lambda_0}Y_{l_1}L^{\lambda_1}\cdots Y_{l_{k-1}}L^{\lambda_{k-1}}\left(\mu_t^{(1-2\epsilon)\Delta}*\mu_t^{\epsilon\Delta-\alpha L}*L^{\lambda_k}\mu_t^{\alpha L}*Y_{l_k}\mu_t^{\epsilon\Delta}\right)(x)\\
&=&\mu_t^{\epsilon\Delta-\alpha L}*L^{\lambda_k}\mu_t^{\alpha L}*Y_{l_k}\mu_t^{\epsilon\Delta}*\left(L^{\lambda_0}Y_{l_1}L^{\lambda_1}\cdots Y_{l_{k-1}}L^{\lambda_{k-1}}\mu_t^{(1-2\epsilon)\Delta}\right)(x).
\end{eqnarray*}
Repeating this decomposition process $k$\ times, we get
%\begin{eqnarray}
%P_L^{l,\lambda}\mu_t^\Delta(x)
%&=&\left(\mu_t^{\epsilon\Delta-\alpha L}*L^{\lambda_k}\mu_t^{\alpha L}*Y_{l_k}\mu_t^{\epsilon\Delta}\right)*\cdots*\left(\mu_t^{\epsilon\Delta-\alpha L}*L^{\lambda_1}\mu_t^{\alpha L}*Y_{l_1}\mu_t^{\epsilon\Delta}\right)\notag\\
%&&*\mu_t^{\epsilon \Delta-\alpha L}*L^{\lambda_0}\mu_t^{\alpha L}(x).\label{decomposition}
%\end{eqnarray}
\begin{eqnarray}
\lefteqn{P_L^{l,\lambda}\mu_t^\Delta(x)
=\left(\mu_t^{\epsilon\Delta-\alpha L}*L^{\lambda_k}\mu_t^{\alpha L}*Y_{l_k}\mu_t^{\epsilon\Delta}\right)*\cdots}\notag\\
&&*\left(\mu_t^{\epsilon\Delta-\alpha L}*L^{\lambda_1}\mu_t^{\alpha L}*Y_{l_1}\mu_t^{\epsilon\Delta}\right)*\mu_t^{\epsilon \Delta-\alpha L}*L^{\lambda_0}\mu_t^{\alpha L}(x).\label{decomposition}
\end{eqnarray}

To estimate
\begin{eqnarray*}
\left(\sum_{l\in \mathcal{I}^k}|P_L^{l,\lambda}\mu_t^\Delta(x)|^2\right)^{1/2}=||(P_L^{\cdot,\lambda}\mu_t^\Delta(x))_\cdot||_{l^2},
\end{eqnarray*}
i.e., the $l^2$\ norm of the $\mathcal{I}^k$-indexed vector $(P_L^{\cdot,\lambda}\mu_t^\Delta(x))_\cdot$, we use the following bound: for any $l^2$\ vector functions ${\bf u}(x)=(u_i(x))_i$, ${\bf v}(x)=(v_j(x))_j$, where each $u_i,v_j\in C(G)$, 
\begin{eqnarray}
\label{minkowskibound}
||{\bf u}*{\bf v}(x)||_{l^2}\leq (||{\bf u}||_{l^2}*||{\bf v}||_{l^2})(x).
\end{eqnarray}
To see this, we use the Minkowski inequality twice as follows
\begin{eqnarray*}
\lefteqn{\left|\left|{\bf u}*{\bf v}(x)\right|\right|^2_{l^2}= \sum_i\left(\sum_j\left|u_i*v_j(x)\right|^2\right)}\\
&=&\sum_i\left|\left|u_i*{\bf v}(x)\right|\right|^2_{l^2}\leq \sum_i \left(|u_i|*\left|\left|{\bf v}\right|\right|_{l^2}(x)\right)^2\leq \left(||{\bf u}||_{l^2}*||{\bf v}||_{l^2}(x)\right)^2.
\end{eqnarray*}
Repeatedly applying (\ref{minkowskibound}), we get
\begin{eqnarray}
\lefteqn{||(P_L^{\cdot,\lambda}\mu_t^\Delta(x))_\cdot||_{l^2}}\notag\\
&\leq& \mu_t^{\epsilon\Delta-\alpha L}*|L^{\lambda_k}\mu_t^{\alpha L}|*||(Y_{l_k}\mu_t^{\epsilon\Delta})_{l_k}||_{l^2}*\cdots*\mu_t^{\epsilon\Delta-\alpha L}*|L^{\lambda_1}\mu_t^{\alpha L}|\notag\\
&&*||(Y_{l_1}\mu_t^{\epsilon\Delta})_{l_1}||_{l^2}*\mu_t^{\epsilon \Delta-\alpha L}*|L^{\lambda_0}\mu_t^{\alpha L}|(x),\label{normdecomposition}
\end{eqnarray}
where each component is a continuous function on $(0,\infty)\times G$. The same is true if we replace $\mu_t^{\epsilon\Delta-\alpha L}$, $\mu_t^{\alpha L}$, $\mu_t^{\epsilon \Delta}$\ in (\ref{normdecomposition}) by any of their time derivatives. Hence $M^{N}_L(\partial_t^p\mu_t^\Delta)=M^{N}_L(G,\,\partial_t^p\mu_t^\Delta)<\infty$\ for any $N,p\in \mathbb{N}$\ and any $t>0$, and $\mu_t^\Delta$\ belongs to $C^\infty((0,\infty)\rightarrow \mathcal{T}_{\Delta,L}(G))$.

Finally, by decomposing the semigroup $(\mu_t^L)_{t>0}$\ into convolutions similarly as above we get that $(\mu_t^L)_{t>0}$\ admits a continuous density function which further belongs to $C^\infty((0,\infty)\rightarrow\mathcal{T}_{\Delta,L}(G))$. For example, for any $k,n\in \mathbb{N}$, similar to (\ref{decomposition}), we can fix some $0<\beta<c/2$\ ($c$\ as in $c\mathcal{E}_\Delta\leq \mathcal{E}_L$) and decompose as follows.
\begin{eqnarray}
\lefteqn{\hspace{-.3in}L^{\lambda_0}Y_{l_1}L^{\lambda_1}\cdots Y_{l_k}L^{\lambda_k}\Delta^{n}\mu_t^L=L^{\lambda_0}Y_{l_1}L^{\lambda_1}\cdots Y_{l_k}\left(\Delta^n\mu_t^{\beta \Delta}*\mu_t^{\frac{1}{2}L-\beta\Delta}*L^{\lambda_k}\mu_t^{\frac{1}{2}L}\right)}\notag\\
&\hspace{-.4in}=& \hspace{-.2in} \Delta^n\mu_t^{\frac{1}{2}\beta\Delta}*\mu_t^{\frac{1}{2}L-\beta\Delta}*L^{\lambda_k}\mu_t^{\frac{1}{2}L}*L^{\lambda_0}Y_{l_1}L^{\lambda_1}\cdots Y_{l_{k-1}}L^{\lambda_{k-1}}\mu_t^{\frac{1}{2}\beta\Delta}. \label{decomposition2}
\end{eqnarray}
The last term is then in the form of (\ref{decomposition}). We can thus bound the terms $\left(\sum_{l\in \mathcal{I}^k}|P_L^{l,\lambda}\Delta^n\partial_t^p\mu_t^L(x)|^2\right)^{1/2}$\ as was done for $\mu_t^\Delta$.
\end{proof}
Next we show that $\mu_t^L$\ satisfies similar off-diagonal Gaussian estimates as $\mu_t^\Delta$\ does, the latter's estimates are obtained in \cite{Gaussian}.
\begin{proposition}
\label{prop1}
Let $G$\ be a compact connected metrizable group. Let $(\mu_t^\Delta)_{t>0}$\ be a symmetric central Gaussian semigroup on $G$\ with generator $-\Delta$. Suppose $(\mu_t^\Delta)_{t>0}$\ satisfies Property (CK$*$). Let $L$\ be any form-comparable perturbation of $\Delta$. Then for any compact set $K\subset G$\ with $e\notin K$, for any $T>0$, $N\in \mathbb{N}$, $\sigma\geq 0$, $A,\alpha>0$,
\begin{eqnarray}
\label{offdiagonal}
\sup_{0<t<T}\frac{e^{AM_L(\alpha t)}}{t^\sigma}M^{N}_{\Delta,L}(K,\,\mu^L_t)<+\infty.
\end{eqnarray}
Here $M_L(t):=\log{\mu_t^{L}(e)}$.
\end{proposition}
\begin{proof}
As in (\ref{decomposition2})(\ref{normdecomposition}) of the previous theorem, the terms
\begin{eqnarray*}
\left(\sum_{l\in \mathcal{I}^k} |P_{L}^{l,\lambda}\Delta^n\mu_t^L(x)|^2\right)^{1/2}=||(P_{L}^{\cdot,\lambda}\Delta^n\mu_t^L(x))_\cdot||_{l^2}
\end{eqnarray*}
are bounded above by convolutions of terms of the form $|D_x^1\Delta^n\mu_t^\Delta|_L$\ (equivalent to $|D_x^1\Delta^n\mu_t^\Delta|_\Delta$), $|L^{|\lambda|}\mu_t^L(x)|$, and $\mu_t^{\theta L-\epsilon \Delta}(x)$, where $\theta,\epsilon\in \mathbb{R}$\ are proper numbers such that $\mathcal{E}_{\theta L-\epsilon \Delta}$\ is comparable with $\mathcal{E}_\Delta$.
The terms $\mu_t^\Delta$, $\mu_t^L$, $\mu_t^{\theta L-\epsilon\Delta}$\ satisfy the following estimates.
\begin{itemize}
    \item[(i)] By \cite{Gaussian}, $\mu_t^\Delta$\ satisfies that for any compact set $K\subset G$\ with $e\notin K$, for any $T>0$, $N\in \mathbb{N}$, $\sigma\geq 0$, $A,\alpha>0$,
\begin{eqnarray}
\label{gaussian1}
\sup_{0<t<T}\frac{e^{AM_\Delta(\alpha t)}}{t^\sigma}M^{N}_\Delta(K,\,\mu^\Delta_t)<+\infty,
\end{eqnarray}
where $M_\Delta(t):=\log{\mu_t^{\Delta}(e)}$.
\item[(ii)] For the term $\mu_t^{L}$, by Remark \ref{CKremark}, $\mu_t^{L}$\ satisfies Property (CK$*$) and $\mu_t^L(e)\leq \mu_t^{\beta\Delta}(e)$\ for some $\beta>0$. Then
\begin{eqnarray*}
M_L(t)=\log{\mu_t^{L}(e)}\leq \log{\mu_t^{\beta\Delta}(e)}=o\left(\frac{1}{t}\right).
\end{eqnarray*}
By \cite{orange}, for any $\psi\in \mathcal{B}(G)$\ with $|D_x^1\psi|_L\leq 1$, $|L\psi|\leq 1$,
\begin{eqnarray*}
\mu_t^{L}(x)\leq \exp{\left\{M_L(t)-\frac{c_L(\psi(x)-\psi(e))^2}{t}\right\}}
\end{eqnarray*}
for some $c_L>0$. It follows that for any compact set $K$\ with $e\notin K$,
\begin{eqnarray*}
\lim\limits_{t\rightarrow 0}\sup\limits_{x\in K}\mu_t^L(x)=0.
\end{eqnarray*}
By \cite{nonGaussian}, $\mu_t^{L}$\ further satisfies that for any $a\in \mathbb{N}$,
\begin{eqnarray}
\label{gaussian2}
\lim_{t\rightarrow 0}\sup_{x\in K}|\partial_t^a\mu_t^{L}(x)|=\lim_{t\rightarrow 0}\sup_{x\in K}|L^a\mu_t^{L}(x)|=0.
\end{eqnarray}
\item[(iii)] Applying the same arguments to $\mu_t^{\theta L-\epsilon\Delta}$\ as to $\mu_t^L$\ (because $\mathcal{E}_{\theta L-\epsilon\Delta}$\ is comparable to $\mathcal{E}_\Delta$) gives
\begin{eqnarray}
\label{gaussian3}
\lim_{t\rightarrow 0}\sup_{x\in K}|\mu_t^{\theta L-\epsilon\Delta}(x)|=0.
\end{eqnarray}
\end{itemize}
By (\ref{gaussian1})(\ref{gaussian2})(\ref{gaussian3}), (\ref{offdiagonal}) follows from the following lemma.
\begin{lemma}
If two families of functions $(u_t)_{t>0}$\ and $(v_t)_{t>0}$\ both satisfy the following off-diagonal bound: for any compact set $K$\ with $e\notin K$,
\begin{eqnarray*}
\lim_{t\rightarrow 0}\sup_{x\in K}u_t(x)=\lim_{t\rightarrow 0}\sup_{x\in K}v_t(x)=0,
\end{eqnarray*}
then $u_t*v_t$\ also satisfies this off-diagonal bound.
\end{lemma}
We now prove the lemma. Let $V_1$\ be an open neighborhood of $e$\ with $\overline{V}_1\cap K=\emptyset$. For any $x\in K$, let $U_x$\ be a small open neighborhood of $x$\ that is away from $V_1$. Let $x(U_x^c)^{-1}=\{xy^{-1}\, |\, y\in U_x^c\}$, then $x(U_x^c)^{-1}$\ is away from $e$, i.e. $x(U_x^c)^{-1}\subset (U_e)^c$\ for some open neighborhood $U_e$\ of $e$. Let $V_x\Subset U_x$\ be a smaller open neighborhood of $x$. We may pick $U_e$\ such that for any $\tilde{x}\in V_x$, $\tilde{x}(U_x^c)^{-1}\subset (U_e)^c$. For any such $\tilde{x}$,
\begin{eqnarray}
\lefteqn{u_t*v_t(\tilde{x})= \int_G u_t(\tilde{x}y^{-1})v_t(y)\,d\nu(y)}\notag\\
&=& \int_{y\in (U_x)^c}u_t(\tilde{x}y^{-1})v_t(y)\,d\nu(y)+\int_{y\in U_x}u_t(\tilde{x}y^{-1})v_t(y)\,d\nu(y)\notag\\
&\leq& \sup_{a\in (U_e)^c}|u_t(a)|\cdot ||v_t||_\infty\cdot \nu(G)+||u_t||_\infty\cdot \sup_{y\in U_x}|v_t(y)|\cdot \nu(G). \label{gaussian4}
\end{eqnarray}
Because $K$\ is compact, using some finite cover $\{V_{x_n}\}_{1\leq n\leq N_0}$\ to cover $K$, the estimate $(\ref{gaussian4})$\ then implies that
\begin{eqnarray*}
\lim_{t\rightarrow 0}\sup_{x\in K}|u_t*v_t(x)|=0
\end{eqnarray*}
as desired.
\end{proof}
\begin{remark}
\label{gaussianremark}
As a consequence, for any $0<a<b<\infty$, for any $A,\alpha,T>0$, $\sigma\geq 0$, $N,p\in \mathbb{N}$, for any compact set $K$\ with $e\notin K$,
\begin{eqnarray*}
\sup_{0<\tau<T}\frac{\exp{\{AM_L(\alpha\tau)\}}}{\tau^\sigma}M^{N,p}_{\Delta,L}((a\tau,b\tau)\times K,\,\mu^L)<\infty.
\end{eqnarray*}
\end{remark}
\section{Proof of Theorem \ref{mainthm} - hypoellipticity}
We now use results in the previous section to prove hypoellipticity properties of the parabolic operator $\partial_t+L$. By definition of $\mathcal{T}_{\Delta,L}'-\mathfrak{S}-$hypoellipticity, for any $U\in \mathcal{T}'_{\Delta,L}(I\times G)$\ and $F\in \mathcal{B}'(I\times G)$\ such that $(\partial_t+L)U=F$, for any open subset $I'\times \Omega'\subset I\times G$\ such that $\psi F\in \mathfrak{S}$\ for any $\psi\in \mathcal{B}_c(I'\times \Omega')$, we need to show that $\psi U\in \mathfrak{S}$\ for any $\psi\in \mathcal{B}_c(I'\times \Omega')$.

Recall that in Theorem \ref{mainthm}, $\mathfrak{S}$\ can be $\mathcal{T}^{N,p}_\Delta$, $\mathcal{S}^{N,p}_L$, $\mathcal{C}_X^{N,p}$, where all spaces are on $I\times G$, and $X$\ is any special projective basis such that $\Delta=-\sum X_i^2$. We prove the theorem for the case
\begin{eqnarray*}
\mathfrak{S}=\mathcal{T}_\Delta^\infty(I\times G)=:\mathcal{T}_\Delta(I\times G),
\end{eqnarray*}
the proofs for other cases are very similar.

Fix any $\psi\in \mathcal{B}_c(I'\times \Omega')$. To show that $\psi U\in \mathcal{T}_{\Delta}(I\times G)$, we construct an approximation sequence in $\mathcal{T}_{\Delta}(I\times G)$.
Let $I_0\times \Omega_0$, $I_1\times \Omega_1$, $I_2\times \Omega_2$\ be open sets such that
\begin{eqnarray*}
\mbox{supp}\{\psi\}\subset I_0\times \Omega_0\Subset I_1\times \Omega_1\Subset I_2\times \Omega_2\Subset I'\times \Omega'.
\end{eqnarray*}
Pick some $\eta\in \mathcal{B}(I\times G)$\ with $\eta\equiv 1$\ on $I_1\times \Omega_1$, $\mbox{supp}\{\eta\}\subset I_2\times \Omega_2$. Fix some bump function $\rho\in C_c^\infty((1,2))$\ that satisfies $\rho\geq 0$\ and $\int_\mathbb{R}\rho(t)\,dt=1$. For any $\tau>0$, let $\rho_\tau(t):=\frac{1}{\tau}\rho\left(\frac{t}{\tau}\right)$. Then $\rho_\tau$\ is supported in $(\tau,2\tau)$. Let $c_0$\ be a positive number to be determined at Lemma \ref{convlemma}. For any $(\alpha,\tau)\in [0,1]\times [0,c_0]\setminus \{(0,0)\}$, define
\begin{eqnarray*}
\widetilde{U}_{\alpha,\tau}:=\begin{cases}
(\rho_\alpha\mu^\Delta)\star (\eta U)\star (\rho_\tau \mu^L),\ \mbox{when }(\alpha,\tau)\in (0,1]\times (0,c_0],\\[0.05in]
\widetilde{U}_{0,\tau}=(\eta U)\star (\rho_\tau \mu^L),\ \mbox{when }\tau>0,\,\alpha=0,\\[0.05in]
\widetilde{U}_{\alpha,0}=(\rho_\alpha\mu^\Delta)\star (\eta U),\ \mbox{when }\alpha>0,\,\tau=0.
\end{cases}
\end{eqnarray*}
Recall that we use $\star$\ to emphasize that the convolution is in time and space, see (\ref{convdef}). The two-parameter sequence $\{\psi \widetilde{U}_{\alpha,\tau}\}_{\alpha,\tau}$\ is our approximation sequence. The following lemma shows that $\psi\widetilde{U}_{\alpha,\tau}\in \mathcal{T}_{\Delta}(I\times G)$\ for all $(\alpha,\tau)\in [0,1]\times[0,c_0]\setminus \{(0,0)\}$.
\begin{lemma}
\label{convlemma}
For any $W\in \mathcal{T}'_{\Delta,L}(I\times G)$\ with compact support in $I\times G$, there exists some $0<c_0<1$, such that for any fixed $0<\tau\leq c_0$,
\begin{eqnarray*}
W_\tau:=W\star(\rho_\tau\mu^{L})\in \mathcal{T}_{\Delta}(I\times G).
\end{eqnarray*}
\end{lemma}
\begin{proof}
First note that $\rho_\tau\mu^L\in \mathcal{T}_\Delta((0,2)\times G)$. The convolution $W_\tau$\ can be interpreted as a continuous function
\begin{eqnarray*}
\lefteqn{W_\tau(s,x)=W((t,y)\mapsto \rho_\tau(s-t)\mu_{s-t}^L(x^{-1}y))}\\
&=&W((t,y)\mapsto \rho_\tau(s-t)\mu_{s-t}^L(y^{-1}x)).
\end{eqnarray*}
So the function (shift the time by $\frac{\tau}{2}$\ for $\mu^L$)
\begin{eqnarray*}
w_\tau:\mathbb{R}\times G\rightarrow \mathbb{R},\ (s,x)\mapsto W((t,y)\mapsto \rho_\tau(s-t)\mu_{s-t-\tau/2}^L(y^{-1}x))
\end{eqnarray*}
is continuous. In fact $w_\tau\in C_c^\infty(\mathbb{R}\rightarrow C(G))$. Because $\mu_{s-t}^L=\mu_{s-t-\tau/2}^L*\mu_{\tau/2}^L$,
\begin{eqnarray*}
\lefteqn{W_\tau(s,x)= W((t,y)\mapsto \rho_\tau(s-t)\int_G\mu_{s-t-\tau/2}^L(y^{-1}xz^{-1})\mu_{\tau/2}^L(z)\,d\nu(z))}\\
&=& \int_G W((t,y)\mapsto \rho_\tau(s-t)\mu_{s-t-\tau/2}^L(y^{-1}xz^{-1}))\,\mu_{\tau/2}^L(z)d\nu(z)\\
&=&(w_\tau(s,\cdot)*\mu_{\tau/2}^L)(x).
\end{eqnarray*}
Because $\mu_{\tau/2}^L\in\mathcal{T}_{\Delta}(G)$, the convolution $W_\tau$\ is in $\mathcal{T}_{\Delta}(\mathbb{R}\times G)$. When $\tau>0$\ is small enough, $W_\tau\in \mathcal{T}_\Delta(I\times G)$.
\end{proof}
\begin{remark}
The same proof shows that $W\star(\rho_\alpha\mu^\Delta)=(\rho_\alpha\mu^\Delta)\star W\in \mathcal{T}_{\Delta}(\mathbb{R}\times G)$. The above method further shows that $W\star(\rho_\alpha\mu^L)$\ and $W\star(\rho_\alpha\mu^\Delta)$\ belong to $\mathcal{T}_{\Delta,L}(\mathbb{R}\times G)$, but this fact is not needed in the proof below.
\end{remark}
\begin{remark}
It is not clear if $W\star \rho_\tau\mu^L$\ converges to $W$\ in $\mathcal{T}'_{\Delta,L}(I\times G)$\ (which is convergence in a very weak sense), since for any $\phi\in \mathcal{T}_{\Delta,L}(I\times G)$,
\begin{eqnarray*}
(W\star \rho_\tau\mu^L)(\phi)=W(\phi\star (\check{\rho}_\tau\check{\mu}^L)),
\end{eqnarray*}
and it is not clear if $\phi\star (\check{\rho}_\tau\check{\mu}^L)$\ converges to $\phi$\ in $\mathcal{T}_{\Delta,L}(I\times G)$. See Section \ref{convorderremark}. Due to this reason, we use the two-parameter approximation sequence $\{\psi\widetilde{U}_{\alpha,\tau}\}$\ to approximate $\psi\eta U=\psi U$.
\end{remark}

We now take the following steps to prove the convergence of the approximation sequence $\{\psi\widetilde{U}_{\alpha,\tau}\}_{\alpha,\tau}$.

\paragraph{Step 1.} For any fixed $\tau\in (0,c_0]$, think of $\psi\widetilde{U}_{\alpha,\tau}$\ as a $\mathcal{T}_\Delta(I\times G)$-valued function in $\alpha$, then
\begin{eqnarray*}
\psi\widetilde{U}_{\alpha,\tau}\in C([0,1]\rightarrow \mathcal{T}_{\Delta}(I\times G)).
\end{eqnarray*}
To prove this, it suffices to check that $\widetilde{U}_{\alpha,\tau}$\ is continuous at $\alpha=0$. This is true since $\rho_\alpha\mu^\Delta\star \left((\eta U)\star \rho_\tau \mu^L\right)\rightarrow (\eta U)\star \rho_\tau \mu^L$\ as $\alpha\rightarrow 0$\ in $\mathcal{T}_{\Delta}(I\times G)$. 
\paragraph{Step 2.}
As $\tau\rightarrow 0$, the sequence $\{\psi\widetilde{U}_{\alpha,\tau}\}_{\tau>0}$\ converges uniformly to some function $G_\alpha$\ in $C([0,1]\rightarrow \mathcal{T}_{\Delta}(I\times G))$, since for any $N,p\in \mathbb{N}$,
\begin{eqnarray*}
\sup_{0\leq \alpha\leq 1}\sup_{0<\tau\leq c_0}M^{N,p}_{\Delta}(I_0\times \Omega_0,\, \partial_\tau \widetilde{U}_{\alpha,\tau})\leq \sup_{0<\tau\leq c_0}M^{N,p}_{\Delta}(I_0\times \Omega_0,\, \partial_\tau \widetilde{U}_{0,\tau})<\infty.
\end{eqnarray*}
We prove this fact in Proposition \ref{prop2} below.
\paragraph{Step 3.} For $0<\alpha\leq 1$, $G_\alpha=\psi\widetilde{U}_{\alpha,0}$, since $\left(\rho_\alpha\mu^\Delta\star (\eta U)\right)\star \rho_\tau \mu^L\rightarrow \rho_\alpha\mu^\Delta\star (\eta U)$\ in $C(I\times G)$\ as $\tau\rightarrow 0$.
\paragraph{Step 4.} Because $\mu_t^\Delta$\ commutes with any function in convolution, we have $\widetilde{U}_{\alpha,0}\rightarrow \eta U$\ in $\mathcal{T}'_{\Delta,L}(I\times G)$\ as $\alpha\rightarrow 0$. Hence $G_0=\psi\eta U=\psi U$, and $\psi U\in \mathcal{T}_\Delta(I\times G)$.\\[0.1in]

To complete the proof of Theorem \ref{mainthm}, it remains to verify Step 2.
\begin{proposition}
\label{prop2}
Under the hypotheses in Theorem \ref{mainthm}, using notations introduced at the beginning of this section,
\begin{eqnarray*}
\sup_{0<\tau\leq c_0}M^{N,p}_{\Delta}(I_0\times \Omega_0,\, \partial_\tau \widetilde{U}_{0,\tau})<\infty.
\end{eqnarray*}
\end{proposition}
\begin{proof} For short we write $\widetilde{U}_\tau$\ for $\widetilde{U}_{0,\tau}$\ in the proof. By computation, $\partial_\tau\rho_\tau(r)=-\partial_r\bar{\rho}_\tau(r)$\ where $\bar{\rho}_\tau(r):=\frac{r}{\tau^2}\rho(\frac{r}{\tau})$;
\begin{eqnarray*}
\lefteqn{\partial_\tau(\widetilde{U}_{\tau}(s,x))= U\left((t,y)\mapsto-\eta(t,y)\partial_s\bar{\rho}_\tau(s-t)\mu_{s-t}^L(x^{-1}y)\right)}\\
&=&U\left(-\eta\partial_s(\bar{\rho}_\tau(s-\cdot)\mathcal{L}_{x^{-1}}\mu_{s-\cdot}^L)\right)+ U\left(\eta\bar{\rho}_\tau(s-\cdot)\partial_s\mathcal{L}_{x^{-1}}\mu_{s-\cdot}^L\right).
\end{eqnarray*}
Here $\mathcal{L}$\ represents left-translation, i.e., for any $x\in G$, for any function $f$, $\mathcal{L}_{x}f(y)=f(xy)$.
In the first term, rewrite the function inside $U$\ as
\begin{eqnarray*}
\lefteqn{-\eta(t,y)\partial_s(\bar{\rho}_\tau(s-t)\mathcal{L}_{x^{-1}}\mu_{s-t}^L(y))}\\
&=&\partial_t(\eta(t,y)\bar{\rho}_\tau(s-t)\mathcal{L}_{x^{-1}}\mu_{s-t}^L(y))-\partial_t\eta(t,y)\cdot \bar{\rho}_\tau(s-t)\mathcal{L}_{x^{-1}}\mu_{s-t}^L(y).
\end{eqnarray*}
For the second term, apply $\partial_s\mathcal{L}_{x^{-1}}\mu_{s-t}^L(y)=-\mathcal{L}_{x^{-1}}L\mu_{s-t}^L(y)=-L\mathcal{L}_{x^{-1}}\mu_{s-t}^L(y)$. We get
\begin{eqnarray*}
\lefteqn{\partial_\tau(\widetilde{U}_\tau(s,x))=-((\partial_t+L) U)(\eta \bar{\rho}_\tau(s-\cdot)\mathcal{L}_{x^{-1}}\mu_{s-\cdot}^L)}\\
&&-U\left(\partial_1\eta \cdot \bar{\rho}_\tau(s-\cdot)\mathcal{L}_{x^{-1}}\mu^L_{s-\cdot}\right)+(\eta LU-L(\eta U))(\bar{\rho}_\tau(s-\cdot)\mathcal{L}_{x^{-1}}\mu_{s-\cdot}^L)\\
&:=& I_\tau(s,x)+I\!I_\tau(s,x)+I\!I\!I_\tau(s,x). 
\end{eqnarray*}
Here $\partial_1\eta(t,y):=\partial_t\eta(t,y)$, the partial derivative w.r.t. the first variable. We estimate each part separately.
\paragraph{Estimate of $I_\tau$.} For
\begin{eqnarray*}
I_\tau(s,x)=-((\partial_t+L) U)(\eta \bar{\rho}_\tau(s-\cdot)\mathcal{L}_{x^{-1}}\mu_{s-\cdot}^L)=-(\eta F)\star(\bar{\rho}_\tau\mu^L)(s,x),
\end{eqnarray*}
to use the condition that $\eta F\in \mathcal{T}_{\Delta}(I\times G)$\ with support in $I_2\times \Omega_2$, we use a trick mentioned in \cite{functionspaces,hypoellipticity} which works as follows.
%For any left-invariant vector field $Y\in \mathfrak{g}$, let $\breve{Y}$\ be the right-invariant vector field on $G$\ with $\breve{Y}(e)=Y(e)$. 
Since $\Delta$\ is bi-invariant, $\Delta=-\sum X_i^2=-\sum \breve{X}_i^2$. In \cite{functionspaces} it is shown that 
\begin{eqnarray*}
\lefteqn{M^N_\Delta(f)=\sup_{x\in G}\sup_{\substack{k,m\in \mathbb{N}\\k+2m\leq N}}\left(\sum_{l\in \mathcal{I}^k}|X_{l_1}X_{l_2}\cdots X_{l_k}\Delta^m f(x)|^2\right)^{1/2}}\\
&=&\sup_{x\in G}\sup_{\substack{k,m\in \mathbb{N}\\k+2m\leq N}}\left(\sum_{l\in \mathcal{I}^k}|\breve{X}_{l_1}\breve{X}_{l_2}\cdots \breve{X}_{l_k}\Delta^m f(x)|^2\right)^{1/2}.
\end{eqnarray*}
The same is true for the $M^{N,p}_\Delta$\ norms. Hence for any $0<\tau\leq c_0$, using the ``right-invariant'' expression of $M^{N,p}_\Delta$, we get that
\begin{eqnarray*}
M^{N,p}_\Delta(I_\tau)\leq M^{N,p}_\Delta(\eta F)\star \bar{\rho}_\tau\mu^L\leq 2M^{N,p}_\Delta(\eta F).
\end{eqnarray*}
\paragraph{Estimate of $I\!I_\tau$.} To estimate $M^{N,p}_\Delta(I_0\times \Omega_0,\,I\!I_\tau)$, as in the proof of Lemma \ref{convlemma}, let
\begin{eqnarray}
u_\tau(s,\omega):=U\left((t,y)\mapsto \partial_t\eta(t,y) \bar{\rho}_\tau(s-t)\mu^L_{s-t-\tau/2}(y^{-1}\omega)\right).
\end{eqnarray}
Then
\begin{eqnarray*}
I\!I_\tau(s,x)=-u_\tau(s,\cdot)*\mu_{\tau/2}^L(x).
\end{eqnarray*}
Note that $s\in I_0$, whereas for the function $\partial_t\eta(t,y)\bar{\rho}_\tau(s-t)$\ to be nonzero, we need $t\in I_1^c$\ and $\tau<s-t<2\tau$. So
\begin{eqnarray*}
M^{N,p}_{\Delta}(I_0\times \Omega_0,\,I\!I_\tau)=0\ \mbox{for }0<\tau<\tau_0,
\end{eqnarray*}
where $\tau_0:=\min\left\{\frac{1}{2}d(I_0,I_1^c),\,c_0\right\}$. For $\tau_0\leq \tau\leq c_0$, since $u_\tau\in C^\infty_c(\mathbb{R}\rightarrow C(G))$, by Minkowski's inequality,
\begin{eqnarray*}
\lefteqn{M^{N,p}_{\Delta}(I_0\times\Omega_0,\,I\!I_\tau)}\\
&\leq& \sup_{\substack{0\leq a\leq p\\k+2m\leq N}}\sup_{\lambda\in \Lambda(k,m)}\sup_{(s,x)\in I\times G}\left\{\left|\partial_s^au_\tau\right|*\left(\sum_{l\in \mathcal{I}^k}|P_\Delta^{l,\lambda}\mu_{\tau/2}^L|^2\right)^{1/2}(x)\right\}.
\end{eqnarray*}
Because $U\in \mathcal{T}_{\Delta,L}'(I\times G)$, there exist some $C>0$\ and $N',p'\in \mathbb{N}$\ that depend on $N,p,U,\mbox{supp}\{\eta\},\rho$, such that for any $0\leq a\leq p$,
\begin{eqnarray*}
||\partial_s^au_\tau||_{L^\infty(I\times G)}\leq C\frac{1}{\tau^{p'}}M^{N',p'}_{\Delta,L}(\eta)M^{N',p'}_{\Delta,L}((\tau/2,\,3\tau/2)\times G,\,\mu^L).
\end{eqnarray*}
Here the term $\frac{1}{\tau^{p'}}$\ is from taking derivatives of $\bar{\rho}_\tau$. Hence
\begin{eqnarray*}
\lefteqn{\sup_{0<\tau\leq c_0}M^{N,p}_{\Delta}(I_0\times \Omega_0,\,I\!I_\tau)=\sup_{\tau_0\leq\tau\leq c_0}M^{N,p}_{\Delta}(I_0\times \Omega_0,\,I\!I_\tau)}\\
&\leq& \frac{C\nu(G)}{\tau_0^{p'}}M^{N',p'}_{\Delta,L}(\eta)M^{N',p'}_{\Delta,L}\left(\left(\frac{\tau_0}{2},\,\frac{3c_0}{2}\right)\times G,\,\mu^L\right) \sup_{\tau_0\leq \tau\leq c_0}M^N_{\Delta}(\mu_{\tau/2}^L)<\infty.
\end{eqnarray*}
\paragraph{Estimate of $I\!I\!I_\tau$.} For the last part $I\!I\!I_\tau$, let $\widetilde{V}:=\eta LU-L(\eta U)$. Then $\widetilde{V}$\ is supported away from $I_1\times\Omega_1$, and
\begin{eqnarray*}
I\!I\!I_\tau(s,x)=\widetilde{V}(\Phi\bar{\rho}_\tau(s-\cdot)\mathcal{L}_{x^{-1}}\mu_{s-\cdot}^L)
\end{eqnarray*}
for some hollow-shaped function $\Phi\in \mathcal{B}(I\times G)$. More precisely, for some $J_1\times\Theta_1$, $J_2\times \Theta_2$\ satisfying
\begin{eqnarray*}
I_0\times\Omega_0\Subset J_1\times\Theta_1\Subset J_2\times\Theta_2\Subset I_1\times\Omega_1,
\end{eqnarray*}
$\Phi$\ satisfies
\begin{eqnarray*}
\mbox{supp}\{\Phi\}\subset I'\times\Omega'\setminus J_1\times\Theta_1,\ \Phi\equiv 1\ \mbox{on }I_2\times\Omega_2\setminus J_2\times\Theta_2.
\end{eqnarray*}
%As in Lemma 3.5 in \cite{},
As in the decomposition of $I\!I_\tau$, $I\!I\!I_\tau$\ can be written as the following convolution
\begin{eqnarray*}
I\!I\!I_\tau(s,x)=\left((\Phi\widetilde{V})\star \bar{\rho}_\tau\mathcal{L}_{-\frac{\tau}{2}}\mu^L\right)(s,\cdot)*\mu_{\tau/2}^L(x),
\end{eqnarray*}
where
\begin{eqnarray*}
\left((\Phi\widetilde{V})\star \bar{\rho}_\tau\mathcal{L}_{-\frac{\tau}{2}}\mu^L\right)(s,w)=\widetilde{V}\left((t,y)\mapsto \Phi(t,y)\bar{\rho}_\tau(s-t)\mu_{s-t-\tau/2}^L(y^{-1}w)\right).
\end{eqnarray*}
So by Minkowski's inequality,
\begin{eqnarray*}
\lefteqn{M^{N,p}_{\Delta}(I_0\times\Omega_0,\,I\!I\!I_\tau)}\\
&\hspace{-.2in}\leq&\hspace{-.1in}\sup_{\substack{(s,x)\in I_0\times\Omega_0\\\lambda\in \Lambda(k,m)\\a,k,m}} \left\{\left|\partial_s^a\left((\Phi\widetilde{V})\star \bar{\rho}_\tau\mathcal{L}_{-\frac{\tau}{2}}\mu^L\right)(s,\cdot)\right|*\left(\sum_{l\in \mathcal{I}^k}|P_\Delta^{l,\lambda}(\mu_{\tau/2}^L)|^2\right)^{1/2}(x)\right\},
\end{eqnarray*}
where the supremum is over $\{(a,k,m):0\leq a\leq p,\,k+2m\leq N\}$.
Let $\Theta_0$\ be an open set satisfying $\Omega_0\Subset \Theta_0\Subset \Theta_1$. We split the convolution into two parts,
\begin{eqnarray*}
\lefteqn{\left|\partial_s^a\left((\Phi\widetilde{V})\star \bar{\rho}_\tau\mathcal{L}_{-\frac{\tau}{2}}\mu^L\right)(s,\cdot)\right|*\left(\sum_{l\in \mathcal{I}^k}|P_\Delta^{l,\lambda}(\mu_{\tau/2}^L)|^2\right)^{1/2}(x)}\\
&\hspace{-.2in}=& \hspace{-.15in}\int_{\Theta_0}\left|\partial_s^a\left((\Phi\widetilde{V})\star \bar{\rho}_\tau\mathcal{L}_{-\frac{\tau}{2}}\mu^L\right)(s,y)\right|\left(\sum_{l\in \mathcal{I}^k}|P_\Delta^{l,\lambda}(\mu_{\tau/2}^L)(y^{-1}x)|^2\right)^{1/2}\,d\nu(y)\\
&&\hspace{-.3in}+\int_{\Theta_0^c}\left|\partial_s^a\left((\Phi\widetilde{V})\star \bar{\rho}_\tau\mathcal{L}_{-\frac{\tau}{2}}\mu^L\right)(s,y)\right|\left(\sum_{l\in \mathcal{I}^k}|P_\Delta^{l,\lambda}(\mu_{\tau/2}^L)(y^{-1}x)|^2\right)^{1/2}\,d\nu(y).
\end{eqnarray*}

The first integral is bounded above by
\begin{eqnarray*}
\lefteqn{\nu(G)\cdot ||(\Phi\widetilde{V})\star \bar{\rho}_\tau\mathcal{L}_{-\frac{\tau}{2}}\mu^L||_{C^p(I_0\rightarrow L^\infty(\Theta_0))}M^N_{\Delta}(\mu^L_{\tau/2})}\\
&\leq& \nu(G)\cdot \frac{C_1}{\tau^{p_1}}M^{N_1,p_1}_{\Delta,L}((\tau/2,\,3\tau/2)\times (\Theta_1^c)^{-1}\Theta_0,\,\mu^L)M^{N}_{\Delta}(\mu_{\tau/2}^L)
\end{eqnarray*}
for some $C_1>0$\ and $N_1,p_1\in \mathbb{N}$\ that depend on $N,p,\rho,\Phi$, and the distribution $\widetilde{V}$\ (in other words, on $U$). Here for any $0<\tau<1$,
\begin{eqnarray*}
M^{N}_{\Delta}(\mu_{\tau/2}^L)\leq C_1'\frac{e^{AM(\alpha\tau)}}{\tau^{N}}
\end{eqnarray*}
for some constants $C_1',\alpha,A>0$, $M(s):=\log{\mu_s^L(e)}$. See Proposition \ref{prop1} and \cite{Gaussian}.

The second integral is bounded above by
\begin{eqnarray*}
\lefteqn{\nu(G)\cdot||(\Phi\widetilde{V})\star \bar{\rho}_\tau\mathcal{L}_{-\frac{\tau}{2}}\mu^L||_{C^p(I_0\rightarrow L^\infty(G))}M^N_{\Delta}((\Theta_0^c)^{-1}\Omega_0,\,\mu_{\tau/2}^L)}\\
&\leq& \nu(G)\cdot \frac{C_1}{\tau^{p_1}}M^{N_1,p_1}_{\Delta,L}((\tau/2,\,3\tau/2)\times G,\,\mu^L)M^N_{\Delta}((\Theta_0^c)^{-1}\Omega_0,\,\mu_{\tau/2}^L).
\end{eqnarray*}

Because $\overline{(\Theta_1^c)^{-1}\Theta_0}$\ and $\overline{(\Theta_0^c)^{-1}\Omega_0}$\ do not contain $e$, applying the off-diagonal Gaussian estimate of $\mu^L_t$\ with $\tau/2\leq t\leq 3\tau/2$\ in Proposition \ref{prop1} (see Remark \ref{gaussianremark}) then shows that
\begin{eqnarray*}
\sup_{0<\tau\leq c_0}M^{N,p}_{\Delta}(I_0\times\Omega_0,\,I\!I\!I_\tau)<\infty.
\end{eqnarray*}
Combining the estimates for $I_\tau$, $I\!I_\tau$, $I\!I\!I_\tau$, we conclude that
\begin{eqnarray*}
\sup_{0<\tau\leq c_0}M^{N,p}_{\Delta}(I_0\times \Omega_0,\, \partial_\tau \widetilde{U}_{\tau})<\infty.
\end{eqnarray*}
\end{proof}
\section{Appendix}
In this appendix we prove some equivalence relations between function space (semi)norms. The following lemma justifies the use of only $\{X_i\}$\ in the definition of the $M_{\Delta,L}^{N}$\ norms (\ref{mixednorm}).
\begin{lemma}
\label{comparablelemma}
Let $G$\ be a compact connected metrizable group as before. Let $L_1,L_2$\ be two form-comparable left-invariant sub-Laplacians on $G$. Let $\{X_i\}_{i\in \mathcal{I}}$, $\{Y_i\}_{i\in \mathcal{I}}$\ be two projective bases such that
\begin{eqnarray*}
L_1=-\sum_{i\in \mathcal{I}}X_i^2,\ L_2=-\sum_{i\in \mathcal{I}}Y_i^2.
\end{eqnarray*}
Then for any $k,p\in \mathbb{N}$, the $S_{L_1}^{k,p}$\ and $S_{L_2}^{k,p}$\ (semi)norms are comparable.
\end{lemma}
\begin{proof}
Under the projective basis $\{X_i\}_{i\in \mathcal{I}}$, $L_2=-\sum_{i,j\in \mathcal{I}}a_{ij}X_iX_j$\ for some real symmetric nonnegative coefficient matrix $A=(a_{ij})$. Because $L_1$\ and $L_2$\ are form comparable, $c\mathcal{E}_{L_1}\leq \mathcal{E}_{L_2}\leq C\mathcal{E}_{L_1}$\ for some $c,C>0$, which is equivalent to the condition
\begin{eqnarray}
\label{comparableequiv}
c\sum \xi_i^2\leq\sum a_{ij}\xi_i\xi_j\leq C\sum \xi_i^2,
\end{eqnarray}
for any $\bf{\xi}=(\xi_i)\in \mathbb{R}^{(\mathcal{I})}$.

Because $\{X_i\}_{i\in \mathcal{I}}$\ is a projective basis, there exists a matrix $\{T_{j}^{i}\}_{i,j\in \mathcal{I}}$\ such that $Y_j=\sum_{i\in \mathcal{I}}T^{i}_{j}X_i$.
%In the rest of the paper we follow the convention of using repeated indices to represent summation, so $Y_j=T^{i}_{j}X_i$.
Define a map $T:\mathbb{R}^{(\mathcal{I})}\rightarrow \mathbb{R}^{(\mathcal{I})}$\ as
\begin{eqnarray*}
T(\xi)=T((\xi_i)_{i\in \mathcal{I}})=(\eta_j)_{j\in \mathcal{I}},\ \ \mbox{where }\eta_j=\sum_{i\in \mathcal{I}}T^{i}_{j}\xi_i.
\end{eqnarray*} 
Then $L=-\sum a_{ij}X_iX_j=-\sum Y_j^2$\ implies that
\begin{eqnarray}
\sum_{i,j\in \mathcal{I}} a_{ij}\xi_i\xi_j=\sum_{j\in \mathcal{I}}|\sum_{i\in \mathcal{I}}T_{j}^{i}\xi_i|^2=||T(\xi)||_{l^2}^2,
\end{eqnarray}
and (\ref{comparableequiv}) is thus equivalent to
\begin{eqnarray}
\label{l2tol2}
c\leq ||T||_{l^2\rightarrow l^2}\leq C\label{2}.
\end{eqnarray}

To show the equivalence of the two (semi)norms $S^{k,p}_{L_1}$\ and $S^{k,p}_{L_2}$\ on $I\times G$\ it suffices to show the equivalence of $S^{k}_{L_1}$\ and $S^{k}_{L_2}$\ on $G$. The proof is essentially a change of variable for the $k-$linear form
\begin{eqnarray*}
(X_{i_1},\cdots,X_{i_k})\mapsto D_x^kf(X_{i_1},\cdots,X_{i_k})=X_{i_1}X_{i_2}\cdots X_{i_k}f(x).
\end{eqnarray*}
More precisely, for any $f\in \mathcal{B}(G)$, for any $1\leq r\leq k$,
\begin{eqnarray*}
\lefteqn{\left(\sum_{j_1,\cdots,j_k} |Y_{j_1}Y_{j_2}\cdots Y_{j_k}f(x)|^2\right)^{1/2}}\\
&=&\left(\sum_{j_1,\cdots,j_{r-1},j_{r+1},\cdots,j_k}\sum_{j_r} \left|Y_{j_1}\cdots \left(\sum_{i}T_{j_r}^iX_i\right)\cdots Y_{j_k}f(x)\right|^2\right)^{1/2}\\
&=&\left(\sum_{j_1,\cdots,j_{r-1},j_{r+1},\cdots,j_k}\sum_{j_r} \left|\sum_{i}T_{j_r}^i\cdot Y_{j_1}\cdots (X_i\cdots Y_{j_k}f(x))\right|^2\right)^{1/2}\\
&=&\left(\sum_{j_1,\cdots,j_{r-1},j_{r+1},\cdots,j_k}\left|\left|T\left((Y_{j_1}\cdots Y_{j_{r-1}}X_iY_{j_{r+1}}\cdots Y_{j_k}f(x))_{i\in \mathcal{I}}\right)\right|\right|_{l^2}^2\right)^{1/2}.
\end{eqnarray*}
In the last line, $(Y_{j_1}\cdots Y_{j_{r-1}}X_iY_{j_{r+1}}\cdots Y_{j_k}f(x))_{i\in \mathcal{I}}$\ denotes the vector indexed by $i$\ with entries $Y_{j_1}\cdots Y_{j_{r-1}}X_iY_{j_{r+1}}\cdots Y_{j_k}f(x)$\ (the other indices $j_1,\cdots,j_{r-1}$, $j_{r+1},\cdots,j_k$\ are fixed). Hence by (\ref{l2tol2}), 
\begin{eqnarray*}
\lefteqn{\left(\sum_{j_1,\cdots,j_k} |Y_{j_1}Y_{j_2}\cdots Y_{j_k}f(x)|^2\right)^{1/2}}\\
&\leq& C\left(\sum_{j_1,\cdots,\hat{j_r},\cdots,j_k}||(Y_{j_1}\cdots X_i\cdots Y_{j_k}f(x))_{i}||_{l^2}^2\right)^{1/2}.
\end{eqnarray*}
Repeating this step gives
\begin{eqnarray*}
\left(\sum_{j_1,\cdots,j_k} |Y_{j_1}Y_{j_2}\cdots Y_{j_k}f(x)|^2\right)^{1/2}\leq C^k\left(\sum_{i_1,\cdots,i_k} |X_{i_1}X_{i_2}\cdots X_{i_k}f(x)|^2\right)^{1/2}.
\end{eqnarray*}
Similarly we have
\begin{eqnarray*}
c^k\left(\sum_{i_1,\cdots,i_k} |X_{i_1}X_{i_2}\cdots X_{i_k}f(x)|^2\right)^{1/2}\leq \left(\sum_{j_1,\cdots,j_k} |Y_{j_1}Y_{j_2}\cdots Y_{j_k}f(x)|^2\right)^{1/2}.
\end{eqnarray*}
So the $S^k_{L_1}$\ and $S^k_{L_2}$\ norms, and with time derivatives added in, the $S^{k,p}_{L_1}$\ and $S^{k,p}_{L_2}$\ (semi)norms, are equivalent respectively.
\end{proof}
When $L_1=\Delta$\ is bi-invariant and $L_2=L$\ is some left-invariant form-comparable perturbation of $\Delta$, by repeating the proof of the above lemma with $\Delta$\ and $L$\ inserted in the differential operator chain, we conclude that the norms defined by taking supremum of any of the following are equivalent
\begin{itemize}
    \item[(i)] $\left(\sum |L^{\lambda_0}X_{l_1}L^{\lambda_1}X_{l_2}L^{\lambda_2}\cdots X_{l_k}L^{\lambda_k}\Delta^b f|^2\right)^{1/2}$;
    \item[(ii)] $\left(\sum |L^{\lambda_0}Y_{l_1}L^{\lambda_1}Y_{l_2}L^{\lambda_2}\cdots Y_{l_k}L^{\lambda_k}\Delta^b f|^2\right)^{1/2}$;
    \item[(iii)] $\left(\sum |L^{\lambda_0}X_{l_1}Y_{l_1'}L^{\lambda_1}X_{l_2}Y_{l_2'}L^{\lambda_2}\cdots X_{l_k}Y_{l_k'}L^{\lambda_k}\Delta^b f|^2\right)^{1/2}$.
\end{itemize}

\bibliographystyle{plain}
\bibliography{infdimgroup}

\end{document}